\documentclass[11pt]{amsart}
\usepackage{amsmath,amssymb,amsthm,amsfonts,setspace,hyperref,dsfont,yhmath}

\newcommand{\NN}{\mathbb{N}}
\newcommand{\RR}{\mathbb{R}}

\newcommand{\CC}{\mathbb{C}}

\newcommand{\ZZ}{\mathbb{Z}}
\newcommand{\GG}{\mathbb{G}}

\newcommand{\norm}[1]{\lVert#1\rVert}
\newcommand{\abs}[1]{\lvert#1\rvert}

\newtheorem{theorem}{Theorem}[section]
\newtheorem{corollary}[theorem]{Corollary}

\newtheorem{lemma}[theorem]{Lemma}
\newtheorem{proposition}[theorem]{Proposition}

\newcommand{\comment}[1]{}
\theoremstyle{definition}
\newtheorem{definition}[theorem]{Definition}

\numberwithin{equation}{section}

\makeatletter
\renewcommand*\env@matrix[1][*\c@MaxMatrixCols c]{%
  \hskip -\arraycolsep
  \let\@ifnextchar\new@ifnextchar
  \array{#1}}
\makeatother
\begin{document}
\title[Twisted geodesic flow]
{The twisted cohomological equation over the partially hyperbolic flow}

\author{ Zhenqi Jenny Wang$^1$ }
\thanks{ $^1$ Based on research supported by NSF grant   DMS-1700837}
\address{Department of Mathematics\\
        Michigan State University\\
        East Lansing, MI 48824,USA}
\email{wangzq@math.msu.edu}

\email{}


\begin{abstract}
Let $\mathbb{G}$ be a higher-rank connected semisimple Lie group
with finite center and without compact factors. In any unitary representation $(\pi, \mathcal{H})$ of $\mathbb{G}$ without non-trivial $\mathbb{G}$-fixed vectors,
we study the twisted cohomological equation $(X+m)f=g$, where $m\in\RR$ and $X$ is in a $\RR$-split Cartan subalgebra of $\text{Lie}(\mathbb{G})$. We characterize the obstructions to solving the
cohomological equation, construct smooth solutions of the cohomological
equation and obtain tame Sobolev estimates for $f$.

We also study common solution to (the infinitesimal version of) the twisted cocycle equation $(X+m)g_1=(\mathfrak{v}+m_1)g_2$, where $\mathfrak{v}$ is nilpotent or in a $\RR$-split Cartan subalgebra, $m,m_1\in\RR$.

This is the first paper studying general twisted equations. Compared to former papers, a new technique in representation theory is developed by Mackey theory and Mellin transform.
\end{abstract}

\maketitle
\section{Introduction}
\subsection{Various algebraic actions} \label{sec:3}We define $\ZZ^k\times \RR^\ell$, $k+\ell\geq 1$ algebraic actions as follows. Let $H$ be a
connected Lie group, $A\subseteq H$ a closed abelian subgroup which
is isomorphic to $\ZZ^k\times \RR^\ell$, $L$ a compact subgroup of
the centralizer $Z(A)$ of $A$, and $\Gamma$ a (cocompact) torsion free lattice in
$H$. Then $A$ acts by left translation on the compact space
$\mathcal{M}=L\backslash H/\Gamma$. Denote this action by $\alpha_A$. The three specific
types of examples discussed below correspond to:
\begin{itemize}
  \item  for the symmetric space examples take $H$ a semisimple Lie group of the non-compact
type.

\smallskip
  \item  for the twisted symmetric space examples take $H=G\ltimes _\rho\RR^m$ or $H=G\ltimes_\rho N$, a semidirect
product of a reductive Lie group $G$ with semisimple factor of the non-compact
type with $\RR^m$ or a simply connected nilpotent group $N$.

\smallskip
  \item for the parabolic action examples, take $H$ a semisimple Lie group of the non-compact
type and $A$ a subgroup of a maximal abelian unipotent subgroup in $H$.
\end{itemize}
\subsection{Motivation and results} In the past two decades various rigidity phenomena for higher rank algebraic actions have been well understood. Significant progresses have been made in the case of cocycle rigidity for both higher-rank
 partially hyperbolic actions (see \cite{Damjanovic1}, \cite{Kononenko},\cite{Spatzier1}, \cite{Zhenqi0}, \cite{W2}) and parabolic actions (see \cite{Forni}, \cite{Ramirez}, \cite{Mieczkowski1}, \cite{W1}).
This is in contrast to the rank-one actions, where Livsic showed that
there is an infinite-dimensional space of obstructions to solving the cohomological
equation for a hyperbolic action by $\RR$ or $\ZZ$. It is natural to extend the study to twisted  cohomological equation and twisted cocycle rigidity.

In fact, twisted  cohomological equation is closely related to prove local differentiable rigidity for algebraic actions by KAM scheme. The KAM method was firstly used  in \cite{Damjanovic4} to obtain local rigidity for genuinely higher-rank partially hyperbolic actions on torus. An adapted version of the scheme was applies to prove weak local rigidity for certain
parabolic algebraic actions on homogeneous space of $SL(2,\RR)\times SL(2,\RR)$ in \cite{DK-parabolic}. A key step in the scheme is to solve the
linearized equation:
\begin{align*}
 \text{Ad}(\alpha)\Omega-\Omega\circ\alpha=0,
\end{align*}
where $\alpha$ is an $A$-algebraic action and $\Omega$ takes values in the tangent space of the homogeneous space. The equation splits into the twisted cohomological equations of the form
\begin{align}\label{for:7}
 \lambda\Omega_i-\Omega_i\circ\alpha=0
\end{align}
on the $\lambda$-eigenvector space of $\text{Ad}(\alpha)$.

To carry out the scheme, we need to make a detailed study of the twisted equation and obtain twisted cocycle rigidity. More precisely, the
construction of the solution to the twisted coboundary equation, classification of
the obstruction and obtaining tame estimates of the solution are needed.

In this paper, we give a complete solution to the twisted cohomological equation over the flow of $\RR$-semisimple element of the Lie algebra. We also obtain twisted cocycle rigidity of the flows of two commuting elements
of the Lie algebra: one nilpotent, and the other $\RR$-semisimple or both are $\RR$-semisimple.
All the results in the present paper are essential for successful application of the KAM-scheme to various algebraic action models in the future work, see \cite{Zhenqi4}, \cite{Zhenqi5}, \cite{Zhenqi6}.

\subsection{History and method} In \cite{Damjanovic4} Damjanovic and Katok use Fourier analysis to prove twisted cocycle rigidity for higher-rank partially hyperbolic actions on torus. In \cite{Spatzier1}
Katok and Spatzier use harmonic analysis method to obtain cocycle rigidity for irreducible Anosov homogeneous actions, which was further developed by the author to extend the results to partially hyperbolic
actions in \cite{W2}. In \cite{Damjanovic1}, \cite{Kononenko} and \cite{Zhenqi0}, the geometric method has been extensively used to study various partially hyperbolic examples.

The natural difficulty in extending the cocycle
rigidity results to twisted cocycles comes from the construction of distributional solutions.
The success of Fourier analysis method is due to  the fact that matrix coefficients for ergodic partially hyperbolic automorphism on torus have super-polynomial decay. By contrast, in the semisimple and other cases at hand  there is a particular speed of exponential decay of matrix coefficients, however smooth the functions  are, and it is not sufficient to construct distribution solutions for the twisted coboundary equations once the absolute value of the $\lambda$ in \eqref{for:7} is not sufficiently close to $1$. This results in the failure of the harmonic analysis method and geometric method to treat the general twisted case.

For parabolic actions, the question is substantially more difficult. Compared to hyperbolic actions, cocycle rigidity results for parabolic actions have been established for very few models. So far the effective tool is representation theory. Flaminio and Forni used representation theory of $SL(2,\RR)$ in \cite{Forni} to study the cohomological equation over the horocycle flow. The method was
further applied in \cite{Mieczkowski1} and \cite{Ramirez} to obtain cocycle rigidity for some models of higher rank parabolic actions. In \cite{Mieczkowski} Mieczkowski used this method to study the cohomological equation over the geodesic flow.  Recently, Tanis and the author used representation theory of higher-rank simple Lie groups to
establish cocycle rigidity for new models of higher rank parabolic actions \cite{W1} and \cite{tanis}. In general, the unitary dual of
many higher rank almost-simple algebraic groups is not completely classified,
and even when the classification is known, it is too complicated to apply.

In this paper, we use representation theory to study the twisted cohomological equation as well as the twisted cocycle equations. The basic idea is as follows: we obtain Sobolev estimates of the solution of the equation in sufficiently many subgroups such that their Lie algebras span the whole tangent space. By the uniqueness of the solution and elliptic regularity theorem we obtain global Sobolev estimates of the solution. The idea firstly appeared  in \cite{W1} to study continuous parabolic actions of certain models and further applied in \cite{tanis} to study discrete parabolic actions. In these papers, the results rely  heavily on the results of horocycle flow and discrete parabolic action on $SL(2,\RR)$; and the method can only treat certain types of split simple Lie groups.

In the current paper, we use mellin transform to get the spectral decomposition of the hyperbolic flow. Hence we don't require the candidates of subgroups with semisimple part $SL(2,\RR)$; instead,
we consider subgroups isomorphic to $\RR\ltimes\RR$, $\RR\ltimes\RR^2$ and $\RR\times\RR$. We use Mackey theory to study these representations and carry out
explicit calculations by Mellin transform in each irreducible component that may appear in restricted non-trivial representation of the big group $G$. We don't rely on any previous results of $SL(2,\RR)$
and this method can be applied to all semisimple Lie groups with finite center. It is the first time representation theory other than that of $SL(2,\RR)$ has been applied to hyperbolic actions for symmetric space examples.
The method and results are of independent interest and have wide
applicability.

\section{Statement of results}
 In this paper, $\mathbb{G}$ denotes a connected semisimple Lie group
with finite center and without compact factors and $\mathfrak{G}$ denotes its Lie algebra. Let $\mathcal{C}$ be a $\RR$-split Cartan subalgebra.
Fix an inner product $|\cdot|$ on $\mathfrak{G}$. Let $\mathfrak{G}^1$ be the set of unit vectors in $\mathfrak{G}$. Suppose $X\in \mathcal{C}\cap \mathfrak{G}^1$ and $m>0$. $\mathfrak{G}$ has the eigenspace decomposition for $\text{ad}_X$:
\begin{align}\label{for:30}
\mathfrak{G}=\sum_{\mu\in\Phi}\mathfrak{g}^\mu
\end{align}
where $\Phi$ is the set of eigenvalues and $\mathfrak{g}^\mu$ is the eigenspace for eigenvalue $\mu$.

In what follows, $C$ will denote any constant that depends only
  on the given  group $\GG$ and $X$. $C_{x,y,z,\cdots}$ will denote any constant that in addition to the
above depends also on parameters $x, y, z,\cdots$.

\subsection{Results for the twisted cohomological equation}\label{sec:4}
Let $y_0=\max\{\mu: \mu>0,\,\mu\in \Phi\}$. The next theorem gives a complete study to the twisted cohomological equation.
\begin{theorem}\label{th:6}
Suppose $(\pi,\mathcal{H})$ is a unitary representation of $\mathbb{G}$ such that the restriction of $\pi$ to any simple factor of $\GG$ is isolated from the trivial
representation (in the Fell topology) and $g\in \mathcal{H}^s$, $s\in 2\NN\cup 0$, then for the twisted cohomological equation
\begin{align*}
 (X+m)f=g
\end{align*}
we have:
\begin{enumerate}
\item\label{for:88} if $s=0$, the equation has a unique solution $f\in \mathcal{H}$ with
\begin{align*}
 \norm{f}\leq m^{-1}\norm{g};
\end{align*}

\smallskip

  \item\label{for:4} if $s>\frac{m}{y_0}$ and the equation has a solution $f\in \mathcal{H}$ with $(I-\mathfrak{v}^2)^{\text{\tiny$\frac{m}{2y_0}$}}f\in \mathcal{H}$, where $0\neq\mathfrak{v}\in \mathfrak{g}^{y_0}$, then $f\in \mathcal{H}^{s}$ and satisfies
the Sobolev estimate
\begin{align*}
    \norm{f}_t\leq C_{t,m}\norm{g}_{\max\{t+2,s\}}\qquad 0\leq t\leq s;
\end{align*}

 \item\label{for:84} if $s>\frac{m}{y_0}$ and the equation has a solution $f\in \mathcal{H}^{\text{\tiny$\frac{m}{y_0}$}}$, then $f\in \mathcal{H}^{s}$ and satisfies
the Sobolev estimate
\begin{align*}
    \norm{f}_t\leq C_{t,m}\norm{g}_{\max\{t+2,s\}}\qquad 0\leq t\leq s;
\end{align*}

\smallskip
  \item \label{for:87} if $s>\frac{m}{y_0}$ and $\mathcal{D}(g)=0$ for any $(X-m)$-invariant distribution $\mathcal{D}$, then the equation has a solution $f\in \mathcal{H}^{s}$;

\medskip
\item \label{for:81}if $s\leq\frac{m}{y_0}$, the equation has a solution $f\in \mathcal{H}^{s-2}$ with
the Sobolev estimates
\begin{align*}
    \norm{f}_t\leq C_{t,m}\norm{g}_{t+2}\qquad 0\leq t\leq s-2.
\end{align*}

\end{enumerate}
\end{theorem}
If $m<0$, we can turn to the equation $(-X-m)f=-g$; and the case of $m=0$ was studied in \cite{W2}.

Furthermore, we can obtain uniform Sobolev upper bounds of the solution for the twisted equation for all $\RR$-semisimple vectors in a small neighborhood of $X$. More precisely, we have:

\begin{theorem}\label{th:5}
Suppose $(\pi,\mathcal{H})$ is a unitary representation of $\mathbb{G}$ such that the restriction of $\pi$ to any simple factor of $\GG$ is isolated from the trivial
representation. Also suppose $s\in 2\NN\cup0$, $X\in \mathcal{C}$ and $m_0>0$. Then there exists $\delta(X,s,m_0), s_0(X,m_0)>0$  such that for any $\RR$-semisimple $Y\in \mathfrak{G}$ and any $m_0\in\RR$ with $|X-Y|+|m-m_0|\leq \delta$, for the twisted cohomological equation,
\begin{align*}
 (Y+m)f=g
\end{align*}
where $g\in \mathcal{H}^s$, we have
\begin{enumerate}
\item\label{for:97} if $s=0$, the equation has a unique solution $f\in \mathcal{H}$ with
\begin{align*}
 \norm{f}\leq 2m_0^{-1}\norm{g};
\end{align*}

\smallskip

  \item\label{for:103} if $s>s_0$ and the equation has a solution $f\in \mathcal{H}^{s_0}$, then $f\in \mathcal{H}^{s}$ and satisfies
the Sobolev estimate
\begin{align*}
    \norm{f}_t\leq C_{t,m_0}\norm{g}_{\max\{t+2,s\}}\qquad 0\leq t\leq s;
\end{align*}

\smallskip
  \item \label{for:104} if $s>s_0$ and $\mathcal{D}(g)=0$ for any $(Y-m)$-invariant distribution $\mathcal{D}$, then the equation has a solution $f\in \mathcal{H}^{s}$.

\end{enumerate}
\end{theorem}

\subsection{Results for the twisted cocycle rigidity}

\begin{theorem}\label{th:7}
Suppose $(\pi,\mathcal{H})$ is a unitary representation of $\mathbb{G}$ such that the restriction of $\pi$ to any simple factor of $\GG$ is isolated from the trivial
representation (in the Fell topology) and $g_1,g_2\in \mathcal{H}^s$, $s\in 2\NN\cup 0$, $s>\frac{m}{y_0}$. Also suppose $\mathfrak{u}\in \mathfrak{G}$ is either nilpotent or in a $\RR$-split Cartan algebra. If $[X,\mathfrak{u}]=0$ and
\begin{align*}
 (X+m)g_1=(\mathfrak{u}+m_1)g_2
\end{align*}
where $m_1\in\RR$ then there is a common solution $h\in \mathcal{H}^s$, that is,
\begin{align*}
 (\mathfrak{u}+m_1)h=g_1,\quad\text{and}\quad (X+m)h=g_2
\end{align*}
with estimates
\begin{align*}
    \norm{h}_t\leq C_{t,m}\max\{\norm{g_1}_{\max\{t+2,s\}},\,\norm{g_2}_{\max\{t+2,s\}}\}\qquad 0\leq t\leq s.
\end{align*}
\end{theorem}

\section{Preliminaries on unitary representation theory}
\subsection{ Mackey representation theory} The problem of determining the complete set of equivalence classes of unitary irreducible
representations of a general class of semi-direct product groups has been solved by Mackey \cite{Mac}.
\begin{theorem}\label{th:1}(Mackey theorem, see \cite[Ex 7.3.4]{Zimmer}, \cite[III.4.7]{Margulis}) Let $S$ be a locally compact second countable group and $\mathcal{N}$ be an abelian closed
normal subgroup of $S$. We define the natural action of $S$ on the group of characters $\widehat{\mathcal{N}}$ of the group $\mathcal{N}$ by setting
\begin{align*}
    (s\chi)(\mathfrak{n}):=\chi(s^{-1}\mathfrak{n}s),\qquad s\in S,\,\chi\in \widehat{\mathcal{N}}, \,\mathfrak{n}\in \mathcal{N}.
\end{align*}
Assume that every orbit $S\cdot \chi$, $\chi\in \widehat{\mathcal{N}}$ is locally closed in $\widehat{\mathcal{N}}$. Then for
any irreducible unitary representation $\pi$ of $S$, there is a point $\chi_0\in \widehat{\mathcal{N}}$ with $S_{\chi_0}$
its stabilizer in $S$, a measure $\mu$ on $\widehat{\mathcal{N}}$ and an irreducible unitary representation  $\sigma$ of $S_{\chi_0}$ such that
\begin{enumerate}
  \item $\pi=\text{Ind}_{S_{\chi_0}}^S(\sigma)$,
  \item $\sigma\mid_{\mathcal{N}}=(\dim)\chi_0$,
  \item $\pi(x)=\int_{\widehat{\mathcal{N}}}\chi(x)d\mu(\chi)$, for any $x\in \mathcal{N}$; and $\mu$ is ergodically supported on the orbit $S\cdot \chi_0$.
\end{enumerate}
\end{theorem}

\subsection{Sobolev space and elliptic regularity theorem}\label{sec:17} Let $\pi$ be a unitary representation of a Lie group $G$ with Lie algebra $\mathfrak{g}$ on a
Hilbert space $\mathcal{H}=\mathcal{H}(\pi)$.
\begin{definition}\label{de;1}
For $k\in\NN$, $\mathcal{H}^k(\pi)$ consists of all $v\in\mathcal{H}(\pi)$ such that the
$\mathcal{H}$-valued function $g\rightarrow \pi(g)v$ is of class $C^k$ ($\mathcal{H}^0=\mathcal{H}$). For $X\in\mathfrak{g}$, $d\pi(X)$ denotes the infinitesimal generator of the
one-parameter group of operators $t\rightarrow \pi(\exp tX)$, which acts on $\mathcal{H}$ as an essentially skew-adjoint operator. For any $v\in\mathcal{H}$, we also write $Xv:=d\pi(X)v$.
\end{definition}
We shall call $\mathcal{H}^k=\mathcal{H}^k(\pi)$ the space of $k$-times differentiable vectors for $\pi$ or the \emph{Sobolev space} of order $k$. The
following basic properties of these spaces can be found, e.g., in \cite{Nelson}, \cite{Goodman} and \cite{Robinson}:
\begin{enumerate}
  \item $\mathcal{H}^k=\bigcap_{m\leq k}D(d\pi(Y_{j_1})\cdots d\pi(Y_{j_m}))$, where $\{Y_j\}$ is a basis for $\mathfrak{g}$, and $D(T)$
denotes the domain of an operator on $\mathcal{H}$.

\medskip
  \item $\mathcal{H}^k$ is a Hilbert space, relative to the inner product
  \begin{align*}
    \langle v_1,\,v_2\rangle_{G,k}:&=\sum_{1\leq m\leq k}\langle Y_{j_1}\cdots Y_{j_m}v_1,\,Y_{j_1}\cdots Y_{j_m}v_2\rangle+\langle v_1,\,v_2\rangle
  \end{align*}
  \item The spaces $\mathcal{H}^k$ coincide with the completion of the
subspace $\mathcal{H}^\infty\subset\mathcal{H}$ of \emph{infinitely differentiable} vectors with respect to the norm
\begin{align*}
    \norm{v}_{G,k}=\bigl\{\norm{v}^2+\sum_{1\leq m\leq k}\norm{Y_{j_1}\cdots Y_{j_m}v}^2\bigl\}^{\frac{1}{2}}.
  \end{align*}
induced by the inner product in $(2)$. The subspace $\mathcal{H}^\infty$
coincides with the intersection of the spaces $\mathcal{H}^k$ for all $k\geq 0$.

\medskip
\item $\mathcal{H}^{-k}$, defined as the Hilbert space duals of
the spaces $\mathcal{H}^{k}$, are subspaces of the space $\mathcal{E}(\mathcal{H})$ of distributions, defined as the
dual space of $\mathcal{H}^\infty$.
  \end{enumerate}
We write $\norm{v}_{k}:=\norm{v}_{G,k}$ and $ \langle v_1,\,v_2\rangle_{k}:= \langle v_1,\,v_2\rangle_{G,k}$ if there is no confusion. Otherwise,
we use subscripts to emphasize that the regularity is measured with respect to $G$.

If $G=\RR^n$ and $\mathcal{H}=L^2(\RR^n)$, the set of square integrable functions on $\RR^n$, then $\mathcal{H}^k$ is the space consisting of all functions on $\RR^n$ whose first $s$ weak derivatives are functions in $L^2(\RR^n)$. In this case, we use the notation $W^k(\RR^n)$ instead of $\mathcal{H}^k$ to avoid confusion. For any open set $\mathcal{O}\subset\RR^n$, $\norm{\cdot}_{(C^r,\mathcal{O})}$ stands for $C^r$ norm for functions having continuous derivatives up to order $r$ on $\mathcal{O}$. We also write $\norm{\cdot}_{C^r}$ if there is no confusion.

We list the well-known elliptic regularity theorem which will be frequently
used in this paper (see \cite[Chapter I, Corollary 6.5 and 6.6]{Robinson}):
\begin{theorem}\label{th:4}
Fix a basis $\{Y_j\}$ for $\mathfrak{g}$ and set $L_{2m}=\sum Y_j^{2m}$, $m\in\NN$. Then
\begin{align*}
    \norm{v}_{2m}\leq C_m(\norm{L_{2m}v}+\norm{v}),\qquad \forall\, m\in\NN
\end{align*}
where $C_m$ is a constant only dependent on $m$ and $\{Y_j\}$.
\end{theorem}

\subsection{Direct decompositions of Sobolev space}\label{sec:3}
For any Lie group $G$ of type $I$ and its unitary representation $\rho$, there is a decomposition of $\rho$ into a direct integral
\begin{align}\label{for:50}
 \rho=\int_Z\rho_zd\mu(z)
\end{align}
of irreducible unitary representations for some measure space $(Z,\mu)$ (we refer to
\cite[Chapter 2.3]{Zimmer} or \cite{Margulis} for more detailed account for the direct integral theory). All the operators in the enveloping algebra are decomposable with respect to the direct integral decomposition \eqref{for:50}. Hence there exists for all $s\in\RR$ an induced direct
decomposition of the Sobolev spaces:
\begin{align}\label{for:51}
\mathcal{H}^s=\int_Z\mathcal{H}_z^sd\mu(z)
\end{align}
with respect to the measure $d\mu(z)$.

The existence of the direct integral decompositions
\eqref{for:50}, \eqref{for:51} allows us to reduce our analysis of the
cohomological equation to irreducible unitary representations. This point of view is
essential for our purposes.
\section{Explicit calculations based on Mackey theory}\label{sec:10}
Suppose $X\in \mathcal{C}$ and $u_1,\,u_2\in \mathcal{B}$ such that
\begin{align}\label{for:23}
 [X,u_i]=\lambda_i u_i, \, i=1,\,2 \quad \text{ and }\quad [u_1,u_2]=0,
\end{align}
where $\lambda_1\lambda_2\neq0$.  Let $S$ denote the connected subgroup with Lie algebra $\{X,\,u_1,\,u_2\}$ and $G$ denote the connected subgroup with Lie algebra $\{X,\,u_1\}$.
\subsection{Unitary dual of $G$} Let $G_1$ denote the connected subgroup of $G$ with Lie algebra spanned by $\{u_1\}$. Then $G_1$ is a normal subgroup of $G$ and $G$ is isomorphic to $\RR\ltimes\RR$.
The group action is defended by
\begin{align*}
  &\text{\small$\exp(\log s_1\cdot X)\big( \exp(\log s_2\cdot X),\, \exp(t_1u_1)\big)\exp(-\log s_1\cdot X)$}\\
  &=\text{\small$\big( \exp(\log s_2\cdot X),\, \exp(s_1^{\lambda_1} t_1u_1)\big)$},
\end{align*}
for any $s_i>0$ and $t_1\in\RR$.

This allows us to completely determine the orbits of the dual action of the group $G_2=\{\exp(\log s\cdot X)\}_{s>0}$ on $G_1$ and the corresponding
representations. The orbits fall into three classes:
\begin{enumerate}
  \item the origin and its stabilizer is $G_2$;
\smallskip

 \item $\{\exp(tu_1):t>0\}$ and for the typical point $\exp(u_1)$ its stabilizer is trivial;

 \smallskip

 \item $\{\exp(tu_1):t<0\}$ and for the typical point $\exp(-u_1)$ its stabilizer is trivial.

\end{enumerate}
The first factors to a representation of $G_2$, which means that $G_1$ acts trivially. Then by using Theorem \ref{th:1} we have
\begin{lemma}
The irreducible representations of $G$ without non-trivial $G_1$-fixed vectors are induced representations and the group action is defined by:
  \begin{gather*}
\beta^\delta: G\rightarrow \mathcal{B}(\mathbb{E}^\delta)\\
\beta^\delta(\exp(\log s\cdot X),0)f(r)=f(s^{-1}r) \\
\beta^\delta(e,\exp(t_1u_1))f(r)=e^{\small\text{$\sqrt{-1} (-1)^\delta t_1r^{-\lambda_1}$}}f(r),
\end{gather*}
for any $t_1\RR$ and $s>0$, where $\delta\in \{+,\,-\}$ and $\lambda_1$ is given by \eqref{for:23}; and
\begin{align*}
 \norm{f}_{\mathbb{E}^\delta}=\|f\|_{L^2(\RR^+, \frac{1}{r}dr)}.
\end{align*}
Computing derived representations, we get
\begin{align}\label{for:41}
 X=-r\partial_r,\qquad u_1=(-1)^\delta r^{-\lambda_1}\sqrt{-1}.
\end{align}
\end{lemma}

\subsection{Unitary dual of $S$}
Let $S_1$ denote the connected subgroup with Lie algebra $\{u_1,\,u_2\}$. Then $S_1$ is a normal subgroup of $S$; and
$S$ and $S_1$ are isomeric to $\RR\ltimes\RR^2$ and $\RR^2$ respectively. The group action is defended by
\begin{align*}
  &\text{\small$\exp(\log s_1\cdot X)\big( \exp(\log s_2\cdot X),\, \exp(t_1u_1+t_2u_2)\big)\exp(-\log s_1\cdot X)$}\\
  &=\text{\small$\big( \exp(\log s_2\cdot X),\, \exp(s_1^{\lambda_1} t_1u_1+s_1^{\lambda_2} t_2u_2)\big)$},
\end{align*}
for any $s_i>0$ and $t_i\in\RR$, $i=1,\,2$.

This allows us to completely determine the orbits of the dual action of the group $S_2=\{\exp(\log s\cdot X)\}_{s>0}$ on $S_1$ and the corresponding
representations. The orbits fall into five classes:
\begin{enumerate}
  \item the origin and its stabilizer is $S_2$;

\smallskip
  \item $\{\exp(tu_1):t>0\}$ or $\{\exp(tu_1):t<0\}$, and for the typical point $\exp(u_1)$ or $\exp(-u_1)$ the stabilizer is trivial;

  \smallskip
  \item $\{\exp(tu_2):t>0\}$ or $\{\exp(tu_2):t<0\}$, and for the typical point $\exp(u_2)$ or $\exp(-u_2)$ the stabilizer is trivial;

  \smallskip
  \item $\{\exp(s^{\lambda_1} u_1+s^{\lambda_2} s_0u_2): s>0,\,s_0\neq 0\}$, and for the typical point $\{\exp(u_1+s_0u_2): s>0,\,s_0\neq 0\}$ its stabilizer is trivial;

\smallskip

 \item $\{\exp(-s^{\lambda_1} u_1+s^{\lambda_2} s_0u_2): s>0,\,s_0\neq 0\}$, and for the typical point $\{\exp(-u_1+s_0u_2): s>0,\,s_0\neq 0\}$ its stabilizer is trivial.
\end{enumerate}
The first factors to a representation of $S_2$, which means that $S_1$ acts trivially; the second corresponds to representations with the group $\{\exp(\log tu_2)\}_{t\in\RR}$ acts trivially; and the the third corresponds to representations with the group $\{\exp(\log tu_1)\}_{t\in\RR}$ acts trivially. Then by using Theorem \ref{th:1} we have
\begin{lemma}\label{le:15}
The irreducible representations of $S$ without non-trivial $S_3$ or $S_4$-fixed vectors are induced representations and parameterized by
$s_0\in \RR\backslash0$ and the group action is defined by:
  \begin{gather*}
\beta^\delta_{s_0}: S\rightarrow \mathcal{B}(\mathbb{E}^\delta_{s_0})\\
\beta^\delta_{s_0}(\exp(\log s\cdot X),0)f(r)=f(s^{-1}r) \\
\beta^\delta_{s_0}(e,\exp(t_1u_1+t_2u_2))f(r)=e^{\small\text{$\sqrt{-1} ((-1)^\delta t_1r^{-\lambda_1}+s_0t_2r^{-\lambda_2})$}}f(r),
\end{gather*}
for any $t_1,\,t_2\in\RR$ and $s>0$, where $\delta\in \{+,\,-\}$, $\lambda_i$, $i=1,\,2$ is given by \eqref{for:23}; and
\begin{align*}
 \norm{f}_{\mathbb{E}^\delta_{s_0}}=\|f\|_{L^2(\RR^+, \frac{1}{r}dr)}.
\end{align*}
Computing derived representations, we get
\begin{align*}
 X=-r\partial_r,\qquad u_1=(-1)^\delta r^{-\lambda_1}\sqrt{-1},\qquad u_2=s_0 r^{-\lambda_2}\sqrt{-1}.
\end{align*}
\end{lemma}

\subsection{Mellin transform}

We recall some basic properties of Mellin transform.    For any constant $c \in \CC$, the Mellin transform is defined by
\[
\mathcal{M}(h,c)=\frac{1}{\sqrt{2\pi}}\int_0^\infty h(r) r^{c}dr.
\]
 For any $c_1<c_2$ (resp. $c_1\leq c_2$), set
\begin{align*}
  \text{\small$\langle c_1,c_2\rangle=\{z\in \CC: c_1<\text{Re}(z)<c_2\}$},\text{ and }\text{\small$\langle\langle c_1,c_2\rangle\rangle=\{z\in \CC: c_1\leq\text{Re}(z)\leq c_2\}$}.
\end{align*}
We use $\mathcal{F}(h)$ to denote the Fourier transform
\begin{align*}
 \mathcal{F}(h,\omega)=\frac{1}{\sqrt{2\pi}}\int_{\RR}h(x)e^{\text{\tiny$-x\omega\sqrt{-1}$}}dx.
\end{align*}
If $c=\alpha+\beta\sqrt{-1}$, $\alpha,\,\beta\in \RR$, then we have
\begin{align}\label{for:49}
 \mathcal{M}(h,c)=\sqrt{2\pi}\mathcal{F}(h(e^{-x})e^{-\alpha x},\beta).
\end{align}
Set $\mathbb{E}=L^2(\RR^+, \frac{1}{r}dr)$. The above relation shows that for any $h\in \mathbb{E}$
\begin{equation}\label{eq:mellin_unitary}
\|h\|^2_{\mathbb{E}} = \int_\RR |\mathcal{M}(h,0+s\sqrt{-1})|^2ds\,.
\end{equation}
Note that $\mathbb{E}^\delta=\mathbb{E}^\delta_{s_0}=\mathbb{E}$. Hence the norm in $\mathbb{E}^\delta$ or $\mathbb{E}^\delta_{s_0}$ is equivalent to the norm defined above.

For any function $\varphi$ defined on the strip $\langle\langle c_1,c_2\rangle\rangle$ and $c_1\leq a\leq c_2$, the Mellin inversion formula is given by
\begin{align}\label{for:63}
h(r)=\mathcal{M}^{-1}(\varphi)=\frac{1}{\text{\small$2\pi\sqrt{-1}$}}\int_{\text{\tiny$a-\sqrt{-1}\infty$}}^{\text{\tiny$a+\sqrt{-1}\infty$}}r^{-z}\varphi(z)dz.
\end{align}
In fact we have
\begin{align}\label{for:64}
\mathcal{M}^{-1}(\varphi)(e^{-x})e^{-ax}=\frac{1}{\text{\small$\sqrt{2\pi}$}}\mathcal{F}^{-1}(\varphi_a,x).
\end{align}
where $\varphi_a(t)=\varphi(a+t\sqrt{-1})$.

\begin{lemma}\label{le:14} If $f\in \mathbb{E}$ and $f\cdot r^{-a}\in \mathbb{E}$, $a>0$, then
\begin{enumerate}
  \item \label{for:56} $\mathcal{M}(f,c)$ is analytic on $\langle -a,0\rangle$;

  \medskip
  \item\label{for:57} if $r\partial_rf\in \mathbb{E}$ and $r\partial_rf\cdot r^{-a}\in \mathbb{E}$, then
  \begin{align*}
   \mathcal{M}(r\partial_rf,c)=-c\mathcal{M}(f,c)
  \end{align*}
  if $c\in \langle\langle -a,0\rangle\rangle$.

\end{enumerate}
\end{lemma}
\begin{proof} We note that
\begin{align}\label{for:58}
  \norm{f\cdot r^{-b}}_{\mathbb{E}}\leq \norm{f\cdot r^{-a}}_{\mathbb{E}}+\norm{f}_{\mathbb{E}}
\end{align}
for any $0\leq b\leq a$.

From \eqref{for:49}, we see that $c\in\CC$ is in the definition strip of $\mathcal{M}(f,\cdot)$ if $c\in\langle\langle -a,0\rangle\rangle$. For any $-a<\alpha<0$ there exists $\epsilon>0$ such that $\langle \alpha-\epsilon,\alpha+\epsilon\rangle\subseteq \langle -a,0\rangle$.  Then for any $n\geq0$ we have
\begin{align}\label{for:54}
 &\int_{\RR}|f(e^{-x})e^{-\alpha x}x^n|^2dx\notag\\
 &=\int_{\RR}|f(r)r^{\alpha }(\log r)^n|^2\frac{1}{r}dr\notag\\
 &\leq C_{n,\epsilon}\int_{\RR}|f(r)r^{\alpha+\epsilon }|^2\frac{1}{r}dr
 +C_{n,\epsilon}\int_{\RR}|f(r)r^{\alpha-\epsilon}|^2\frac{1}{r}dr.
\end{align}
This shows that $f(e^{-x})e^{-\alpha x}\cdot x^n\in L^2(\RR, dx)$. Hence by \eqref{for:49} we have
\begin{align}\label{for:55}
 \frac{\partial^n}{\partial\beta^n}\mathcal{M}(f,\alpha+\beta\sqrt{-1})&=\sqrt{2\pi}\frac{\partial^n}{\partial\beta^n}\mathcal{F}(h(e^{-x})e^{-\alpha x},\beta)\notag\\
 &=(-\sqrt{-1})^n\sqrt{2\pi}\cdot \mathcal{F}(f(e^{-x})e^{-\alpha x}x^n,\beta),
\end{align}
for any $n\geq0$.

On the other hand, let
\begin{align*}
 \Delta_{\alpha,t}(f)(x)=t^{-1}(f(e^{-x})e^{-(\alpha+t) x}-f(e^{-x})e^{-\alpha x}).
\end{align*}
Then we have
\begin{align*}
 \lim_{t\to0}\Delta_{\alpha,t}(f)(x)=-f(e^{-x})e^{-\alpha x}x
\end{align*}
for all $x\in\RR$; and
\begin{align*}
|\Delta_{\alpha,t}(f)|\leq |f(e^{-x})e^{-(\alpha-\epsilon/2)x} x|+|f(e^{-x})e^{-(\alpha+\epsilon/2)x} x|
\end{align*}
if $\abs{t}\leq\epsilon/2$. Similar to \eqref{for:54}, we can show that
\begin{align}\label{for:65}
 |f(e^{-x})e^{-(\alpha-\epsilon/2)} x^n|+|f(e^{-x})e^{-(\alpha+\epsilon/2)} x^n|\in L^2(\RR, dx)
\end{align}
for any $n\geq 0$.

Then by dominated convergence theorem we have
\begin{align*}
 \lim_{t\to0}\Delta_{\alpha,t}(f)=-f(e^{-x})e^{-\alpha x}x \qquad \text{in }L^2(\RR, dx).
\end{align*}
Since $\mathcal{F}$ is isometric by \eqref{for:49} we have
\begin{align*}
 \partial_\alpha\mathcal{M}(f,\alpha+\beta\sqrt{-1})&=\sqrt{2\pi}\lim_{t\to0}\Delta_{\alpha,t}\big(\mathcal{F}(f(e^{-x})e^{-\alpha x},\beta)\big)\\
 &=\sqrt{2\pi}\mathcal{F}\big(\lim_{t\to0}\Delta_{\alpha,t}(f(e^{-x})e^{-\alpha x}),\beta\big)\\
 &=-\sqrt{2\pi}\mathcal{F}(f(e^{-x})e^{-\alpha x}x,\beta).
\end{align*}
Compared with \eqref{for:55}, we see that
\begin{align*}
 \sqrt{-1}\partial_\alpha\mathcal{M}(f,\alpha+\beta\sqrt{-1})=\partial_\beta\mathcal{M}(f,\alpha+\beta\sqrt{-1}).
\end{align*}
If we can show that $\mathcal{M}(f,z)\in C^1(\langle -a,0\rangle)$ (in the sense that $\langle -a,0\rangle$ is viewed as a subset of $\RR^2$), then we finish the proof of \eqref{for:56}.

Substituting  $f$ by $f\cdot\log r$ and repeating the above process we have
\begin{align}\label{for:68}
 \frac{\partial^2}{\partial\alpha^2}\mathcal{M}(f,\alpha+\beta\sqrt{-1})&=-\sqrt{2\pi}\lim_{t\to0}\Delta_{\alpha,t}\big(\mathcal{F}(f(e^{-x})e^{-\alpha x}x,\beta)\big)\notag\\
 &\overset{\text{(1)}}{=}-\sqrt{2\pi}\mathcal{F}\big(\lim_{t\to0}\Delta_{\alpha,t}(f(e^{-x})e^{-\alpha x}x),\beta\big)\notag\\
 &=\sqrt{2\pi}\mathcal{F}(f(e^{-x})e^{-\alpha x}x^2,\beta).
\end{align}
Here in $(1)$ we use that
\begin{align*}
 \lim_{t\to0}\Delta_{\alpha,t}(f\cdot\log r)=f(e^{-x})e^{-\alpha x}x^2 \qquad \text{in }L^2(\RR, dx).
\end{align*}
Moreover, for $0\leq j\leq 2$ we have
\begin{align}\label{for:69}
  &\int_{c=\alpha-\frac{\epsilon}{2}}^{\alpha+\frac{\epsilon}{2}}\int_{\beta=-\infty}^\infty|\mathcal{F}(f(e^{-x})e^{-c x}x^j,\beta)|^2 d\beta dc\notag\\
  &=\int_{c=\alpha-\frac{\epsilon}{2}}^{\alpha+\frac{\epsilon}{2}}\int_{r=0}^\infty|f(r)r^{c}(\log r)^j|^2 \frac{1}{r}dr dc\notag\\
  &\overset{\text{(1)}}{\leq} C_\epsilon\int_{c=\alpha-\frac{\epsilon}{2}}^{\alpha+\frac{\epsilon}{2}}\big(\norm{f\cdot r^{c+\frac{\epsilon}{2}}}_{\mathbb{E}}^2+\norm{f\cdot r^{c-\frac{\epsilon}{2}}}_{\mathbb{E}}^2\big) dc\notag\\
  &\overset{\text{(2)}}{\leq} C_\epsilon\int_{c=\alpha-\frac{\epsilon}{2}}^{\alpha+\frac{\epsilon}{2}}(\norm{f\cdot r^{-a}}^2_{\mathbb{E}}+\norm{f}^2_{\mathbb{E}})dc\notag\\
  &\leq C_\epsilon(\norm{f\cdot r^{-a}}^2_{\mathbb{E}}+\norm{f}^2_{\mathbb{E}}).
\end{align}
Here in $(1)$ we use \eqref{for:54}; and in $(2)$ we use \eqref{for:58}.

By \eqref{for:55} and \eqref{for:68}, it follows from Elliptic regularity theorem that $\mathcal{M}(f,z)\in W^2(\langle \alpha-\epsilon,\alpha+\epsilon\rangle)$ (see Section \ref{sec:17}); which implies that
$\mathcal{M}(f,z)\in C^1(\langle \alpha-\epsilon,\alpha+\epsilon\rangle)$
by Sobolev embedding theorem. Hence we see that $\mathcal{M}(f,z)\in C^1(\langle -a,0\rangle)$.

\medskip
\eqref{for:57}: If $r\partial_rf\in \mathbb{E}$ and $r\partial_rf\cdot r^{-a}\in \mathbb{E}$, then by \eqref{for:49} for $c=\alpha+\beta\sqrt{-1}\in \langle\langle -a,0\rangle\rangle$ we have
\begin{align*}
 \mathcal{M}(r\partial_rf,c)&\overset{\text{(1)}}{=}\sqrt{2\pi}\mathcal{F}(\partial_rf(e^{-x})\cdot e^{-(\alpha+1) x},\beta)\\
 &=\sqrt{2\pi}\mathcal{F}\Big(-\partial_x\big(f(e^{-x})\cdot e^{-\alpha x}\big)-\alpha f(e^{-x})\cdot e^{-\alpha x},\beta\Big)\\
 &= -(\alpha+\beta\sqrt{-1})\sqrt{2\pi}\mathcal{F}(f(e^{-x})\cdot e^{-\alpha x},\beta)\\
 &= -c\mathcal{M}(f,c).
\end{align*}
Here in $(1)$ we used the fact that $r\partial_rf\cdot r^{-b}\in \mathbb{E}$ for any $0\leq b\leq a$, see \eqref{for:58}.  Hence we get \eqref{for:57}.
\end{proof}
Following exactly the same proof line we can show that
\begin{corollary}\label{cor:2}
If $f\in \mathbb{E}$ and $f\cdot r^{a}\in \mathbb{E}$, $a>0$, then
\begin{enumerate}
  \item \label{for:93} $\mathcal{M}(f,c)$ is analytic on $\langle 0,a\rangle$;

  \medskip
  \item\label{for:94} if $r\partial_rf\in \mathbb{E}$ and $r\partial_rf\cdot r^{a}\in \mathbb{E}$, then
  \begin{align*}
   \mathcal{M}(r\partial_rf,c)=-c\mathcal{M}(f,c)
  \end{align*}
  if $c\in \langle\langle 0,a\rangle\rangle$.

\end{enumerate}

\end{corollary}

\section{Twisted coboundary for the PH  flow of $\mathbb{G}$}
The goal of the paper is to study the twisted equation
\begin{align}\label{for:39}
  (X+m)f=g
\end{align}
in a unitary representation of $\mathbb{G}$. The basic idea is as follows:
 \begin{enumerate}
   \item We show that the solution to equation \eqref{for:39} unique in any unitary representation of $\mathbb{G}$, see Lemma \ref{le:12}. This allows us to  study equation \eqref{for:39} in various subgroups of
    $\mathbb{G}$; and thus obtain bounded derivatives in these subgroups. We will choose sufficiently many such subgroups that their Lie algebras span the whole tangent space, therefore the global Sobolev estimates follow from the elliptic regularity theorem, see Theorem \ref{th:4}.

    \smallskip
   \item By unitary representation of abelian groups we obtain bounded derivatives in $\mathfrak{g}^\phi$ with $\phi=0$ as well as in $\mathfrak{D}$ (\eqref{for:30}), see Lemma \ref{le:12}.

   \smallskip
   \item The remaining directions are in $\mathfrak{g}^\phi$ with $\phi\neq0$. Then we consider subgroups in $\mathbb{G}$ with Lie algebras spanned by $\{X, v\}$, $v\in\mathfrak{g}^\phi$ with $\phi\neq0$; or
   $\{X, v,u\}$, $v\in\mathfrak{g}^\phi$ and $u\in\mathfrak{g}^\psi$ with $\phi>0$ and $\psi>0$. Theses subgroups are isomorphic to $\RR\ltimes\RR$ or $\RR\ltimes\RR^2$. Recall the following direct consequence of the well known Howe-Moore theorem
on vanishing of the matrix coefficients at infinity \cite{howe-moore}: if $(\pi,\mathcal{H})$ denotes a unitary representation of $\mathbb{G}$ such that the restriction of $\pi$ to each simple factor has no non-trivial fixed vectors, then $\pi$ has
no $M$-invariant vector for any closed non-compact subgroup $M$ of $\mathbb{G}$. By Howe-Moore, we only need to consider unitary representations of $\RR\ltimes\RR$ or $\RR\ltimes\RR^2$ computed in Section \ref{sec:10}. We carry out explicit computation in these unitary representations to obtain upperbounds of derivatives in these subgroups, see  Theorem \ref{th:2} and Proposition \ref{po:1}.
 \end{enumerate}

\subsection{Twisted coboundary for a flow in any Lie groups} In this part we present several technical results which are important for the subsequent discussion.

\begin{lemma}\label{le:12}
Suppose $G$ is a Lie group and $(\pi,\mathcal{H})$ is a unitary representation of $G$. Also suppose $0\neq\mathfrak{u}\in \text{Lie}(G)$ such that the one-parameter subgroup $\{\exp(t\mathfrak{u})\}_{t\in\RR}$
is isomorphic to $\RR$. Then:
\begin{enumerate}
  \item\label{for:89} for any $g\in \mathcal{H}$ and any
$s\in\RR\backslash 0$, the twisted equation
\begin{align*}
 (\mathfrak{u}+s)f=g
\end{align*}
has a unique solution $f\in \mathcal{H}$ with
\begin{align*}
  \norm{f}\leq \abs{s}^{-1}\norm{g};
\end{align*}

\smallskip
  \item\label{for:90} if $Y\in \text{Lie}(G)$ with $[Y,\mathfrak{u}]=0$ and $Y^ng\in \mathcal{H}$, $n\in\NN$, then $Y^nf\in \mathcal{H}$ with
  \begin{align*}
  \norm{Y^nf}\leq \abs{s}^{-1}\norm{Y^ng}.
\end{align*}

\end{enumerate}

\end{lemma}
\begin{proof}
\eqref{for:89}: For the one-parameter subgroup $\{\exp(t\mathfrak{u})\}_{t\in\RR}$ we have a direct integral decomposition
\begin{align*}
    \pi\mid_{\exp(t\mathfrak{u})}=\int_{\widehat{\RR}}\chi(t)du(\chi)
\end{align*}
where $u$ is a regular Borel measure and
\begin{align*}
    v=\int_{\widehat{\RR}}v_{\chi}du(\chi),\qquad \forall\, v\in \mathcal{H}.
\end{align*}
Set
\begin{align*}
  f_{\chi}=(s+\chi'(0))^{-1}g_{\chi},\qquad \chi\in \widehat{\RR}.
\end{align*}
We see that $f=\int_{\widehat{\RR}}(s+\chi'(0))^{-1}g_{\chi}du(\chi)$ is a formal solution of the equation $(\mathfrak{u}+s)f=g$.

Next, we will show that $f\in \mathcal{H}$.  Since $\chi'(0)\in i\RR$,
\begin{align}\label{for:96}
 |s+\chi'(0)|\geq \abs{s},\qquad \forall\,\chi\in \widehat{\RR}.
\end{align}
Then
\begin{align*}
 \norm{f}^2=\int_{\widehat{\RR}}\abs{s+\chi'(0)}^{-2}\norm{g_{\chi}}^2du(\chi)\leq \abs{s}^{-2}\int_{\widehat{\RR}}\norm{g_{\chi}}^2du(\chi)=\abs{s}^{-2}\norm{g}^2.
\end{align*}
This shows that $f\in \mathcal{H}$.

On the other hand, if $(\mathfrak{u}+s)f=0$ with $f\in \mathcal{H}$, then we have
\begin{align*}
  (s+\chi'(0))f_{\chi}=0
\end{align*}
for almost every $\chi\in \widehat{\RR}$ with respect to $u$. Then from \eqref{for:96} we see that $f_{\chi}=0$ for almost every $\chi\in \widehat{\RR}$. This means that $f=0$. Hence we showed the uniqueness of the solution of the twisted equation. This completes the proof.

\medskip
\eqref{for:90} We consider the connected subgroup $S=\{\exp(tY+r\mathfrak{u})\}_{t,r\in\RR}$. Since one-parameter subgroup $\{\exp(tY)\}_{t\in\RR}$ is either isomorphic to $\RR$ or isomorphic to $S^1$, $S$ is
either isomorphic to $\RR^2$ or isomorphic to $\RR\times S^1$. Then we have a direct integral decomposition
\begin{align*}
    \pi\mid_{\exp(t\mathfrak{u}+rY)}=\int_{\widehat{\RR^2}}\chi(t)\eta(r)du(\chi,\eta)
\end{align*}
if $S=\RR^2$; or
\begin{align*}
    \pi(\exp(t\mathfrak{u}),k)=\int_{\widehat{\RR\times S^1}}\chi(t)\eta(k)du(\chi,\eta),\qquad k\in S^1
\end{align*}
if $S=\RR\times S^1$, where $u$ is a regular Borel measure.

 We can write
\begin{align*}
    v=\int_{\widehat{\RR^2}}v_{\chi,\eta}du(\chi,\eta)\quad\text{or}\quad v=\int_{\widehat{\RR\times S^1}}v_{\chi,\eta}du(\chi,\eta),\qquad \forall\, v\in \mathcal{H}.
\end{align*}
Set
\begin{align*}
  f_{\chi,\eta}=(s+\chi'(0))^{-1}g_{\chi,\eta}.
\end{align*}
It is clear that
\begin{align*}
 \eta'(0)^\ell f_{\chi,\eta}=\eta'(0)^\ell(s+\chi'(0))^{-1} g_{\chi,\eta}, \qquad (\chi,\eta)\in \widehat{\RR^2}
\end{align*}
or
\begin{align*}
 (m\sqrt{-1})^\ell f_{\chi,\eta}=(m\sqrt{-1})^\ell(s+\chi'(0))^{-1} g_{\chi,\eta}, \qquad (\chi,\eta)\in \widehat{\RR\times S^1}
\end{align*}
where $\eta(e^{\theta\sqrt{-1}})=e^{m\theta\sqrt{-1}}$, $m\in\ZZ$, $0\leq \theta<2\pi$ for any $0\leq \ell\leq n$.

This implies that
\begin{align*}
 \norm{Y^\ell f}^2=\int_{\widehat{\RR}}\abs{s+\chi'(0)}^{-2}\norm{Y^\ell g_{\chi,\eta}}^2du(\chi,\eta)\leq \abs{s}^{-2}\norm{Y^\ell g}^2.
\end{align*}
for any $0\leq \ell\leq n$. Hence we finish the proof.
\end{proof}

\subsection{Twisted coboundary for the HP flow in irreducible component of $\RR\ltimes\RR$ and $\RR\ltimes\RR^2$} We assume notations in Section \ref{sec:10}.

\begin{lemma}\label{for:40}
Suppose $\lambda_1>0$ and $s>0$. In any irreducible representation $(\beta^\delta,\,\mathbb{E}^\delta)$ of $G$, if $g\in (\mathbb{E}^\delta)^s$, then
for any $\frac{\epsilon}{\lambda_1}\leq a\leq s-\frac{\epsilon}{\lambda_1}$, where $\epsilon$ is sufficiently small we have
 \begin{align}\label{for:42}
|\mathcal{M}(g,-a\lambda_1+t\sqrt{-1})|\leq C_{\epsilon}\norm{g}_{a+\frac{\epsilon}{\lambda_1}},\qquad \forall\,t\in\RR,
\end{align}
 Hence the linear functional
\begin{align*}
 \mathcal{D}_{\delta,m}(h)=\mathcal{M}(h,-m)
\end{align*}
is defined for any $h\in (\mathbb{E}^\delta)^\ell$, if $\ell>\frac{m}{\lambda_1}$ satisfying
\begin{align}\label{for:45}
|\mathcal{D}_{\delta,m}(h)|\leq C_{\epsilon, \lambda_1}\norm{h}_{\frac{m+\epsilon}{\lambda_1}}
\end{align}
where $\frac{\epsilon}{\lambda_1}\leq\frac{m}{\lambda_1}\leq \ell-\frac{\epsilon}{\lambda_1}$; moreover, $\mathcal{D}_{\delta,m}(h)$ is an $(X-m)$-invariant distribution.
\end{lemma}
\begin{proof}
Since $g\in (\mathbb{E}^\delta)^s$, by \eqref{for:41} we see that $g\cdot r^{-s\lambda_1}\in \mathbb{E}^\delta$. By \eqref{for:55} and \eqref{for:54}, for $0\leq n\leq 1$ we have
\begin{align}\label{for:74}
&\int_\RR |\frac{\partial^n}{\partial t^n}\mathcal{M}(g,-a\lambda_1+t\sqrt{-1})|^2dt\leq C_\epsilon\norm{g}_{a+\frac{\epsilon}{\lambda_1}}.
\end{align}
Hence \eqref{for:42}  follows from Sobolev imbedding theorem. It is clear that \eqref{for:45} is a direct consequence of  \eqref{for:42}.

If $g=(x+m)f$, where $f\in (\mathbb{E}^\delta)^\infty$  then it follows from Lemma \ref{le:14} that
\begin{align*}
 \mathcal{M}(g,-m)=\mathcal{M}\big((-r\partial_rf+m)f,-m\big)=(-m+m) \mathcal{M}(f,-m)=0.
\end{align*}
This shows that $\mathcal{D}_{\delta,m}(h)$ is $(X-m)$-invariant. Hence we finish the proof.
\end{proof}
The next theorem a crucial step in proving Theorem \ref{th:6}. \eqref{for:53} is the most difficult part of the proof. The scheme of the proof of \eqref{for:53} is as follows:
\begin{enumerate}
  \item By Mellin transform we construct the formal solution \eqref{for:77} on the strip $\langle\langle -s\lambda_1,0\rangle\rangle$ and obtain $L^2$ norm of the formal along each vertical line in the strip, see \emph{Step I};

  \smallskip
  \item by using  Mellin inversion theorem, in \emph{Step II} we show that the formal solution corresponds to a solution $f$ with $(I-u_1^2)^\frac{t}{2}f\in \mathbb{E}^\delta$, $0<t<s$;

  \smallskip
  \item since $\norm{(I-u_1^2)^\frac{t}{2}f}$ are uniformly bounded near $t=0$ and near $t=s$, we can show that $(I-u_1^2)^\frac{t}{2}f\in \mathbb{E}^\delta$ for both $t=0$ and $t=s$, see \emph{Step III}.
\end{enumerate}

\begin{theorem}\label{th:2}
Suppose $s>0$. In any irreducible representation $(\beta^\delta,\,\mathbb{E}^\delta)$ of $G$, if $g\in (\mathbb{E}^\delta)^s$, then
\begin{enumerate}

  \item\label{for:52}if $\lambda_1>0$ and  if $s>\frac{m}{\lambda_1}$ and the equation
\eqref{for:39} has a solution $f\in \mathbb{E}^\delta$ with $(I-u_1^2)^{\text{\tiny$\frac{m}{2\lambda_1}$}} f\in \mathbb{E}^\delta$ then $\mathcal{D}_{\delta,m}(g)=0$;

\medskip
  \item\label{for:53} if $\lambda_1>0$, $s>\frac{m}{\lambda_1}$ and $\mathcal{D}_{\delta,m}(g)=0$, the equation
\eqref{for:39} has a solution $f\in \mathbb{E}^\delta$ with estimates
\begin{align*}
 \norm{(I-u_1^2)^\frac{t}{2}f}&\leq \left\{\begin{aligned} &C_\epsilon\norm{g}_{t},&\qquad &\text{if }|t\lambda_1-m|\geq \epsilon/2,\\
&C_{\epsilon}\norm{g}_{\frac{m+\epsilon}{\lambda_1}},&\qquad &\text{if }|t\lambda_1-m|<\epsilon/2,
\end{aligned}
 \right.
 \end{align*}
 for any $0\leq t\leq s$, where $\epsilon<\min\{\frac{m}{2},\frac{1}{2},\frac{s\lambda_1-m}{2}\}$;

\medskip
\item\label{for:76} if $\lambda_1>0$ and $s\leq \frac{m}{\lambda_1}$ the equation
\eqref{for:39} has a solution $f\in \mathbb{E}^\delta$ with $u_1^tf\in \mathbb{E}^\delta$, for any $0\leq t<s$ satisfying
\begin{align*}
  \norm{(I-u_1^2)^\frac{t}{2}f}\leq C(m-t\lambda_1)^{-1}\norm{g}_t;
\end{align*}

\medskip
\item \label{for:92}if $\lambda_1<0$ the equation
\eqref{for:39} has a solution $f\in \mathbb{E}^\delta$ with $(I-u_1^2)^\frac{s}{2}f\in \mathbb{E}^\delta$ satisfying
\begin{align*}
  \norm{(I-u_1^2)^\frac{t}{2}f}\leq C(m-t\lambda_1)^{-1}\norm{g}_t,\qquad 0\leq t\leq s.
\end{align*}
\end{enumerate}
\end{theorem}

\begin{proof}

\eqref{for:52}: By assumption we have $f\cdot r^{-m},\,g\cdot r^{-m}\in \mathbb{E}^\delta$. Since $f$ is the solution of the equation \eqref{for:39}, $Xf\cdot r^{-m}$ is also in $\mathbb{E}^\delta$. Hence for almost all $t\in\RR$ we have
\begin{align*}
 \mathcal{M}(g,-m+t\sqrt{-1})&=\mathcal{M}((X+m)f,-m+t\sqrt{-1})\\
 &\overset{\text{(1)}}{=}t\sqrt{-1}\mathcal{M}(f,-m+t\sqrt{-1}).
\end{align*}
Here in $(1)$ we used \eqref{for:57} of Lemma \ref{le:14}.

By \eqref{for:56} of Lemma \ref{le:14}, We see that the line $\{-m+t\sqrt{-1}\}_{t\in\RR}$ is inside the analytic strip of $\mathcal{M}(g,\cdot)$. Moreover,
from \eqref{for:49} we have
\begin{align*}
  &\int_{\RR}|t^{-1}\mathcal{M}(g,-m+t\sqrt{-1})|^2dt\\
  &=\int_{\RR}|\mathcal{M}(f,-m+t\sqrt{-1})|^2dt=\norm{f\cdot r^{-m}}^2.
\end{align*}
This shows that $\mathcal{M}(g,-m)=0$. Then we get \eqref{for:52}.

\medskip
\eqref{for:53}: \emph{\textbf{Step I: Construction of distributional solutions}}.

\smallskip
For any $z\in \langle\langle -s\lambda_1,0\rangle\rangle$ set
\begin{align}\label{for:77}
 \mathcal{P}(z)=\frac{\mathcal{M}(g,z)}{m+z}.
\end{align}
By \eqref{for:56} of Lemma \ref{le:14}, $\mathcal{M}(g,\cdot)$ is analytic on the strip $\langle -s\lambda_1,0\rangle$. Note that the line $\{-m+t\sqrt{-1}\}_{t\in\RR}$ is inside the analytic strip.
The the assumption $\mathcal{M}(g,-m)=0$ implies that $\mathcal{P}(z)$ is
also analytic on the strip $\langle -s\lambda_1,0\rangle$. Since the definition strip of $\mathcal{M}(g,\cdot)$ is $\langle\langle -s\lambda_1,0\rangle\rangle$, the definition strip of $\mathcal{P}$ is also $\langle\langle -s\lambda_1,0\rangle\rangle$.

Choose $\epsilon<\min\{\frac{m}{2},\frac{1}{2},\frac{s\lambda_1-m}{2}\}$. Then $\langle\langle -m-\epsilon,-m+\epsilon\rangle\rangle\subset \langle -s\lambda_1,0\rangle$.  It is clear that
\begin{align*}
 |\mathcal{P}(z)|\leq C_\epsilon|\mathcal{M}(g,z)|,\qquad \text{if }|z+m|\geq \epsilon.
\end{align*}
Then $-s\lambda_1\leq a\leq 0$ for we have
\begin{align*}
  &\text{\small$\int_{|a+t\sqrt{-1}-m|\geq\frac{\epsilon}{2}} |\mathcal{P}(a+t\sqrt{-1})|^2dt$}\notag\\
  &\leq C_\epsilon\int_\RR |\mathcal{M}(g,a+t\sqrt{-1})|^2dt=C_\epsilon\norm{u_1^{\frac{|a|}{\lambda_1}}g}
  \leq C_\epsilon\norm{g}_{\frac{|a|}{\lambda_1}}.
\end{align*}
Let $B_{-m}(\epsilon)$ denote the  ball of radius $\epsilon$ centered at $-m$ on the complex plane. Then by maximum modulus principle, if $|z+m|<\epsilon$ we have
\begin{align*}
  \max_{z\in B_{-m}(\epsilon)}|\mathcal{P}(z)|&\leq \max_{z\in \partial B_{-m}(\epsilon)}|\mathcal{P}(z)|\leq C_\epsilon\max_{z\in \partial B_{-m}(\epsilon)}|\mathcal{M}(g,\partial B_{-m}(\epsilon))|\\
  &\overset{\text{(1)}}{\leq} C_{\epsilon}\norm{g}_{\frac{m+2\epsilon}{\lambda_1}}.\notag
\end{align*}
Here $(1)$ follows from Lemma \ref{for:40}.

Hence we have
\begin{align}\label{for:60}
\text{\small$\int_\RR |\mathcal{P}(a+t\sqrt{-1})|^2dt$}\leq&\left\{\begin{aligned} &C_\epsilon\norm{g}_{\frac{|a|}{\lambda_1}},&\qquad &\text{if }|a+m|\geq \epsilon/2,\\
&C_{\epsilon}\norm{g}_{\frac{m+\epsilon}{\lambda_1}},&\qquad &\text{if }|a+m|<\epsilon/2,
\end{aligned}
 \right.
\end{align}
for any $-s\lambda_1\leq a\leq 0$.

From \eqref{for:63} and \eqref{for:64} the function $f_a(r)$ obtained by
\begin{align*}
  f_a(r)&=\frac{1}{\text{\small$2\pi\sqrt{-1}$}}\int_{\text{\tiny$a-\sqrt{-1}\infty$}}^{\text{\tiny$a+\sqrt{-1}\infty$}}r^{-z}\mathcal{P}(z)dz,
\end{align*}
where $-s\lambda_1\leq a\leq 0$ exists with estimates
\begin{align}
 \norm{f_a\cdot r^{a}}^2&=\text{\small$\int_\RR \big|\mathcal{P}(a+t\sqrt{-1})\big|^2dt$}\label{for:1}\\
 &\overset{\text{(1)}}{\leq} \left\{\begin{aligned} &C_\epsilon\norm{g}_{\frac{|a|}{\lambda_1}},&\qquad &\text{if }|a+m|\geq\epsilon/2,\\
&C_{\epsilon}\norm{g}_{\frac{m+\epsilon}{\lambda_1}},&\qquad &\text{if }|a+m|<\epsilon/2, \label{for:61}
\end{aligned}
 \right.
 \end{align}
 for any $-s\lambda_1\leq a\leq 0$. Here in $(1)$ we use \eqref{for:60}.

 Hence we see that these $f_a$, $-s\lambda_1\leq a\leq 0$, are distributions. Next, we will shows that these $f_a$, $-s\lambda_1< a< 0$ are distributional solutions of the equation \eqref{for:39}. Note that
  \begin{align}\label{for:71}
   |z\mathcal{P}(z)|\leq C_{m}  |\mathcal{M}(g,z)|\overset{\text{(2)}}{\leq} C_{m,\epsilon}\norm{g}_{s}
\end{align}
if $\frac{\epsilon}{\lambda_1}\leq -\frac{\text{Re}(z)}{\lambda_1}\leq s-\frac{\epsilon}{\lambda_1}$ and $|z+m|\geq \min\{\frac{m}{2}, \frac{s\lambda_1-m}{2},\frac{1}{2}\}$. Here  $(1)$ follows from of Lemma \ref{for:40}.

Moreover, by noting that $\mathcal{P}(z)$ is analytic on the strip $\langle -s\lambda_1,0\rangle$ we conclude that
 \begin{align}\label{for:70}
   |z\mathcal{P}(z)|\leq C_{m, \epsilon,g}\norm{g}_{s}
\end{align}
if $\frac{\epsilon}{\lambda_1}\leq -\frac{\text{Re}(z)}{\lambda_1}\leq s-\frac{\epsilon}{\lambda_1}$. Since
\begin{align*}
 \mathcal{M}(g)_a(t)=\mathcal{M}(g,a+t\sqrt{-1})\in L^2(\RR,dt),
\end{align*}
for $-s\lambda_1\leq a\leq 0$, \eqref{for:71} and \eqref{for:70} imply that $t\mathcal{P}_a(t)\in L^2(\RR,dt)$ if we let $\mathcal{P}_a(t)=\mathcal{P}(a+t\sqrt{-1})$, $-s\lambda_1< a<0$. From
 \eqref{for:60} we see that $\mathcal{P}_a(t)\in L^2(\RR,dt)$. Then by \eqref{for:64} we have
\begin{align*}
 &\frac{1}{\text{\small$\sqrt{2\pi}$}}\mathcal{F}^{-1}\big((a+t\sqrt{-1})\mathcal{P}_a(t),x\big)\\
 &=a f_a(e^{-x})e^{-ax}+\partial_x\big(f_a(e^{-x})e^{-ax}\big)\\
 &=-\partial_rf_a(e^{-x})e^{-(a+1)x}.
\end{align*}
Hence we have
\begin{align*}
  &-\partial_rf_a(e^{-x})e^{-(a+1)x}+mf_a(e^{-x})e^{-ax}\\
  &=\frac{1}{\text{\small$\sqrt{2\pi}$}}\mathcal{F}^{-1}\big((a+t\sqrt{-1}+m)\mathcal{P}_a(t),x\big)\\
  &\overset{\text{(1)}}{=}\frac{1}{\text{\small$\sqrt{2\pi}$}}\mathcal{F}^{-1}\big(\mathcal{M}(g)_a(t),x\big)\\
  &=g(e^{-x})e^{-ax}.
\end{align*}
Here in $(1)$ we use relation \eqref{for:77}.

This is equivalent to
\begin{align}\label{for:72}
 -\partial_rf_a(r)r+mf_a(r)=g(r).
\end{align}
Hence $f_a$ is a solution of the equation \eqref{for:39}.

\medskip

\emph{\textbf{Step II: Coincidence of $f_a$, $-s\lambda_1< a< 0$}}.

\smallskip

In previous step we showed that both $\mathcal{P}_a(t)$ and $t\mathcal{P}_a(t)$ are in $L^2(\RR,dt)$ for any $-s\lambda_1< a< 0$. This implies that $\mathcal{P}_a(t)\in L^1(\RR)$. Moreover, \eqref{for:71} implies that for any sufficiency small $\epsilon>0$,
$\mathcal{P}_a(t)$ tends to zero uniformly as $t\to \pm\infty$ for any $a\in [-s\lambda_1+\eta,-\eta]$. Then by Mellin inversion theorem, $f_b=f_a$, $-s\lambda_1< a,b< 0$; moreover, letting $f=f_a$, $-s\lambda_1< a< 0$, $f$ is continuous on $(0,\infty)$ and the Mellin transform of $f$ is $\mathcal{P}(z)$ on $\langle -s\lambda_1,0\rangle$.  Then it follows from \eqref{for:61} that
\begin{align}\label{for:73}
 \norm{f\cdot r^{a}}^2& \leq \left\{\begin{aligned} &C_\epsilon\norm{g}_{\frac{|a|}{\lambda_1}},&\qquad &\text{if }|a+m|\geq\epsilon,\\
&C_{\epsilon}\norm{g}_{\frac{m+\epsilon}{\lambda_1}},&\qquad &\text{if }|a+m|<\epsilon,
\end{aligned}
 \right.
 \end{align}
 for any $-s\lambda_1< a< 0$.

\medskip

\emph{\textbf{Step III: Estimates of  $\norm{f}$ and $\norm{f\cdot r^{-s\lambda_1}}$}}.

\smallskip
 By using \eqref{for:64}, from \eqref{for:60} we see that $f_0,\,f_{-s\lambda_1}\in \mathbb{E}^\delta$ with
\begin{align}\label{for:62}
 \norm{f_0}\leq C_\epsilon\norm{g}\quad\text{ and }\quad\norm{f_{-s\lambda_1}}\leq C_\epsilon\norm{g}_s.
\end{align}
Since $g\cdot r^{-c}\to g$ in $\mathbb{E}^\delta$ and $g\cdot r^{-s\lambda_1+c}\to g\cdot r^{-s\lambda_1}$ in $\mathbb{E}^\delta$ as $c\to 0^+$ we have
\begin{align*}
   &\text{\small$\int_\RR \big|\mathcal{M}(g,-c+t\sqrt{-1})-\mathcal{M}(g,0+t\sqrt{-1})\big|^2dt$}\to 0\qquad\text{and}\\
   &\text{\small$\int_\RR \big|\mathcal{M}(g,-s\lambda_1+c+t\sqrt{-1})-\mathcal{M}(g,-s\lambda_1+t\sqrt{-1})\big|^2dt$}\to 0
\end{align*}
as $c\to 0^+$. Then we have
\begin{align*}
 &\int_\RR \big|\mathcal{P}(-c+t\sqrt{-1})-\mathcal{P}(0+t\sqrt{-1})\big|^2dt\to 0\qquad\text{and }\\
 &\int_\RR \big|\mathcal{P}(-s\lambda_1+c+t\sqrt{-1})-\mathcal{P}(-s\lambda_1+t\sqrt{-1})\big|^2dt\to 0
\end{align*}
as $c\to 0^+$,  which implies that
\begin{align*}
 \norm{f\cdot r^{-c}-f_0}^2\to 0 \quad\text{ and }\quad \norm{f\cdot r^{-s\lambda_1+c}-f_{-s\lambda_1}}^2\to 0
\end{align*}
in $\mathbb{E}^\delta$ as $c\to 0^+$. Then there exists a sequence $c_n\to0$ as $n\to \infty$ such that
\begin{align*}
  f\cdot r^{-c_n}\to f_0 \quad\text{ and }\quad f\cdot r^{-s\lambda_1+c_n}\to f_{-s\lambda_1}
\end{align*}
for almost all $r$ (with respect to the the measure $\frac{1}{r}dr$) on $(0,\infty)$.

Since $f\cdot r^{-c}\to f$  and $f\cdot r^{-s\lambda_1+c}\to f\cdot r^{-s\lambda_1}$ as $c\to 0^+$ on $(0,\infty)$,  we get
\begin{align*}
 f_0=f \quad\text{ and }\quad f_{-s\lambda_1}=f\cdot r^{-s\lambda_1} \quad\text{ in }\quad \mathbb{E}^\delta.
\end{align*}
Then \eqref{for:62} shows that
\begin{align*}
 \norm{f}\leq C_\epsilon\norm{g}\quad\text{ and }\quad\norm{f\cdot r^{-s\lambda_1}}\leq C_\epsilon\norm{g}_s.
\end{align*}
This together with \eqref{for:73} give the result.

\medskip
\eqref{for:76}: Define $\mathcal{P}(z)$ on the strip on the strip $\langle\langle -t\lambda_1,0\rangle\rangle$ as in \eqref{for:77} for any $0<t<s$. It is clear that
\begin{align}\label{for:8}
 |\mathcal{P}(z)|\leq |z+m|^{-1}|\mathcal{M}(g,z)|
\end{align}
for any $z\in \langle\langle -t\lambda_1,0\rangle\rangle$.

Arguments in \eqref{for:53} show that $f=f_a$, $0\leq a\leq t\lambda_1$ is well-defined and continuous on $(0,\infty)$; moreover the Mellin transform of $f$ is $\mathcal{P}(z)$ on $\langle\langle -t\lambda_1,0\rangle\rangle$. Then it is clear that $(I-u_1^2)^\frac{t}{2}f\in \mathbb{E}^\delta$ and the estimate follows immediately from \eqref{for:1} and \eqref{for:8}.

\medskip
\eqref{for:92}: By \eqref{for:93} of Corollary \ref{cor:2}, we see that $\mathcal{P}(z)$ as defined in \eqref{for:77} is analytic on the strip $\langle0, -s\lambda_1\rangle$ with estimates
\begin{align}\label{for:9}
\text{\small$\int_\RR |\mathcal{P}(a+t\sqrt{-1})|^2dt$}\leq \frac{1}{m+a}\norm{g}_{-\frac{a}{\lambda_1}}
\end{align}
for any $0\leq a\leq -s\lambda_1$.

By following the same proof line as in \eqref{for:53}, we can show that $f=f_a$, $0\leq a\leq -s\lambda_1$ is well-defined and continuous on $(0,\infty)$; moreover the Mellin transform of $f$ is $\mathcal{P}(z)$ on $\langle\langle 0,-s\lambda_1\rangle\rangle$. Then it is clear that $(I-u_1^2)^\frac{s}{2}f\in \mathbb{E}^\delta$ and the estimate follows from \eqref{for:9} immediately.

\end{proof}

\subsection{Global twisted coboundary for the HP flow in $\RR\ltimes\RR$}\label{sec:8} Let $(\beta,\mathcal{U})$ be a unitary representation of $G$ (see Section \ref{sec:10}) without non-trivial $u_1$-invariant vectors. We now discuss how to obtain a global solution from the solution which exists in each irreducible component of $\mathcal{U}$. By general arguments in Section \ref{sec:3} there is a direct decomposition of
$\mathcal{U}=\int_Z \mathcal{U}_zd\mu(z)$ of irreducible unitary representations of $G$ for some measure space $(Z,\mu)$. If $\beta$ has no non-trivial $u_1$-invariant vectors, then for almost
all $z\in Z$, $\beta_z$ has no non-trivial $u_1$-invariant vectors. This means that for almost
all $z\in Z$ $(\beta_z,\mathcal{U}_z)=(\beta^\delta,\mathbb{E}^\delta)$.   Arguments in Section \ref{sec:3} show that we can apply
Theorem \ref{th:2} to prove the following:
\begin{corollary}\label{cor:1}
Suppose $s>0$. Let $(\beta,\mathcal{U})$ be a unitary representation of $G$ without non-trivial $u_1$-invariant vectors.  If $g\in \mathcal{U}^s$, then
\begin{enumerate}

  \item\label{for:78} if $\lambda_1>0$, $s>\frac{m}{\lambda_1}$ and the equation
\eqref{for:39} has a solution $f\in \mathcal{U}$ with $(I-u_1^2)^{\text{\tiny$\frac{m}{2\lambda_1}$}} f\in \mathcal{U}$,  then $u_1^tf\in \mathcal{U}$ for any $0\leq t\leq s$  with estimates
\begin{align*}
 \norm{(I-u_1^2)^\frac{t}{2}f}&\leq \left\{\begin{aligned} &C_\epsilon\norm{g}_{t},&\qquad &\text{if }|t\lambda_1-m|\geq \epsilon,\\
&C_{\epsilon}\norm{g}_{\frac{m+\epsilon}{\lambda_1}},&\qquad &\text{if }|t\lambda_1-m|<\epsilon,
\end{aligned}
 \right.
 \end{align*}
where $\epsilon<\min\{\frac{m}{2},\frac{1}{2},\frac{s\lambda_1-m}{2}\}$;

\medskip
\item \label{for:86} if $\lambda_1>0$, $s>\frac{m}{\lambda_1}$ and $\mathcal{D}(g)=0$ for any $(X-m)$-invariant distribution $\mathcal{D}$, then the twisted cohomological equation $(X+m)f=g$ has a solution $f\in \mathcal{U}$ with $(I-u_1^2)^\frac{s}{2}f\in \mathcal{U}$;

\medskip
\item\label{for:85}  if $\lambda_1>0$ and $s\leq \frac{m}{\lambda_1}$ the equation
\eqref{for:39} has a solution $f\in \mathcal{U}$ with $(I-u_1^2)^\frac{t}{2}f\in \mathcal{U}$, for any $0\leq t<s$ satisfying
\begin{align*}
  \norm{(I-u_1^2)^\frac{t}{2}f}\leq C(m-t\lambda_1)^{-1}\norm{g}_t;
\end{align*}

\medskip
\item \label{for:95} if $\lambda_1<0$ the equation
\eqref{for:39} has a solution $f\in \mathcal{U}$ with $(I-u_1^2)^\frac{s}{2}f\in \mathcal{U}$ satisfying
\begin{align*}
  \norm{(I-u_1^2)^\frac{t}{2}f}\leq C(m-t\lambda_1)^{-1}\norm{g}_t,\qquad 0\leq t\leq s.
\end{align*}

\end{enumerate}

\end{corollary}
\begin{proof}
The cohomological equation
\eqref{for:39} has a decomposition
\begin{align}\label{for:79}
 (X+m)f_z=g_z
\end{align}
with $g_z\in\mathcal{U}_z^s$ for almost all $z\in Z$. Next we show the proof of \eqref{for:78} for the case of $|t\lambda_1-m|\geq\epsilon$.
The assumption implies that the equation \eqref{for:79} has a solution $f_z\in \mathcal{U}_z$ with $(I-u_1^2)^{\text{\tiny$\frac{m}{2\lambda_1}$}} f_z\in \mathcal{U}_z$ for almost all $z\in Z$.  Then it follows from \eqref{for:52} of Theorem \ref{th:2} that $\mathcal{D}_{\delta,m}(g_z)=0$ for almost all $z\in Z$; moreover, \eqref{for:53} of Theorem \ref{th:2} shows that
\begin{align*}
 \norm{(I-u_1^2)^\frac{t}{2}f_z}\leq C_\epsilon\norm{g_z}_{t}
\end{align*}
almost all $z\in Z$ if $|t\lambda_1-m|\geq\epsilon$. Since $C_\epsilon$ are constants only dependent on $\epsilon$ we have
\begin{align}\label{for:91}
 \norm{(I-u_1^2)^\frac{t}{2}f}^2=\int_Z\norm{(I-u_1^2)^\frac{t}{2}f_z}^2d\mu(z)\leq C_\epsilon \int_Z\norm{g_z}_t^2d\mu(z)=\norm{g}^2_{t}.
\end{align}
This proves of the case of $|t\lambda_1-m|\geq\epsilon$. The other cases follow in exactly the same way. Hence we get \eqref{for:78}, \eqref{for:85} and \eqref{for:95}.

To prove \eqref{for:86}, we note that
the assumption and Lemma \ref{for:40}  implies that for almost all $z\in Z$, $\mathcal{D}_{\delta,m}(g_z)=0$ if $(\beta_z,\mathcal{U}_z)=(\beta^\delta,\mathbb{E}^\delta)$. \eqref{for:53} of Theorem \ref{th:2}
shows that the equation \eqref{for:79} has a solution $f_z\in \mathcal{U}_z$ with $u_1^{s} f_z\in \mathcal{U}_z$ for almost all $z\in Z$ with estimates
\begin{align*}
 \norm{(I-u_1^2)^\frac{s}{2}f_z}\leq C_\epsilon\norm{g_z}_{s}
\end{align*}
for almost all $z\in Z$. Then similar to \eqref{for:91}, we see that $f=\int_Zf_zd\mu(z) \in \mathcal{U}$ with $(I-u_1^2)^\frac{s}{2}f\in \mathcal{U}$. Then we get \eqref{for:86}.

\end{proof}
\subsection{Global common solution for the cocycle equation in $(\RR\ltimes\RR)\times\RR$} In this part, we study the cocycle equation for $G\times\RR$, which will be used to prove Theorem \ref{th:7}.
\begin{proposition}\label{po:3}
Suppose $\lambda_1>0$, $s>\frac{m}{\lambda_1}$. Let $(\beta,\mathcal{U})$ be a unitary representation of $G\times\RR$ without non-trivial $u_1$ or $\RR$-invariant vectors. Let $\chi=1\in\text{Lie}(\RR)$. If $g_1,g_2\in \mathcal{U}^s$, such that
\begin{align*}
 (X+m)g_1=(\chi+m_1) g_2
\end{align*}
$m_1\in\RR$ then there exists $f\in \mathcal{U}$ with $(I-u_1^2)^\frac{s}{2}f\in \mathcal{U}$ such that
\begin{align*}
  (\chi+m_1)f=g_1,\qquad\text{and}\qquad (X+m)f=g_2.
\end{align*}

\begin{proof}
Irreducible unitary representations of $G$ without non-trivial $\RR$-invariant vectors are of the form $\beta^\delta\otimes \zeta_v$, where $\zeta_v$, $v\in\RR$ is an irreducible unitary representation of $\RR$ with the action $\zeta_v(x)=e^{\sqrt{-1}vx}$ for any $x\in\RR$. Arguments in Section \ref{sec:3} allows us to reduce
our analysis of the cocycle equation to each irreducible component $\beta^\delta_\nu\otimes \zeta_v$ appears in $\beta$. Then by assumption, we only need to consider $\beta^\delta_\nu\otimes \zeta_v$, $v\neq0$.

Note that the cocycle equation has the form
\begin{align}\label{for:10}
 (X+m)g_{1,v}=(\sqrt{-1}v+m_1) g_{2,v}
\end{align}
in $\beta^\delta_\nu\otimes \zeta_v$, where $g_{1,v},g_{2,v}\in (\beta^\delta)^s$. Since $g_{1,v}\in (\beta^\delta)^s$, it follows from \eqref{for:52} of Theorem \ref{th:2} that $\mathcal{D}_{\delta,m}(g_{2,v})=0$ (note that
$\sqrt{-1}v+m_1\neq0$). Then \eqref{for:53} of Theorem \ref{th:2} shows that the equation
\begin{align}\label{for:11}
 (X+m)f_v=g_{2,v}
\end{align}
has a solution $f_{v}\in \mathbb{E}^\delta$ with $(I-u_1^2)^\frac{s}{2}f_v\in \mathbb{E}^\delta$ satisfying
\begin{align}\label{for:6}
 \norm{(I-u_1^2)^\frac{s}{2}f_v}\leq C_{s,\lambda_1,m}\norm{g_{2,v}}_s.
\end{align}
From \eqref{for:10} and \eqref{for:11} we immediately have
\begin{align*}
(X+m)g_{1,v}=(\sqrt{-1}v+m_1) (X+m)f_v.
\end{align*}
From Lemma \ref{le:12}, we see that
\begin{align*}
 (\sqrt{-1}v+m_1)f_v=g_{1,v}.
\end{align*}
Since the constant $C_{s,\lambda_1,m}$ in \eqref{for:6} is uniform (independent of the representation $\beta^\delta_\nu\otimes \zeta_v$), hence we have
a common solution $f\in \mathcal{U}$ with $(I-u_1^2)^\frac{s}{2}f\in \mathcal{U}$.
\end{proof}

\end{proposition}
\subsection{Global twisted coboundary for the HP flow in $\RR\ltimes\RR^2$}  In this part, we obtain results in unitary representations of $\RR\ltimes\RR^2$ as described in Section \ref{sec:10}.
We also assume notations in Section \ref{sec:10}.
\begin{proposition}\label{po:1}
Suppose $\lambda_1\geq\lambda_2>0$. In any unitary representation $(\beta,\,\mathcal{E})$ of $S$ without non-trivial $u_1$ or $u_2$-fixed vectors, if $g\in \mathcal{E}^s$, $s>\frac{m}{\lambda_1}$, and the equation
\eqref{for:39} has a solution $f\in \mathcal{E}$ with $(I-u_1^2)^{\text{\tiny$\frac{m}{2\lambda_1}$}} f\in \mathcal{E}$, then $(I-u_2^2)^\frac{s}{2}f\in \mathcal{E}$ with estimates as follows:
\begin{enumerate}
  \item if $s>\frac{m}{\lambda_2}$  then
\begin{align*}
 \norm{(I-u_2^2)^\frac{t}{2}f}&\leq \left\{\begin{aligned} &C_\epsilon\norm{g}_{t},&\qquad &\text{if }|t\lambda_2-m|\geq\epsilon/2,\\
&C_{\epsilon}\norm{g}_{\frac{m+\epsilon}{\lambda_1}},&\qquad &\text{if }|t\lambda_2-m|<\epsilon/2,
\end{aligned}
 \right.
 \end{align*}
where $0\leq t\leq s$ and $\epsilon<\min\{\frac{m}{2},\frac{1}{2},\frac{s\lambda_2-m}{2}\}$;

\medskip
  \item if $s\leq \frac{m}{\lambda_2}$ then
  \begin{align*}
   \norm{(I-u_2^2)^\frac{t}{2}f}\leq C(m-t\lambda_2)^{-1}\norm{g}_t
  \end{align*}
  for any $0\leq t<s$.

\end{enumerate}
\end{proposition}

\begin{proof} By arguments in Section \ref{sec:8}, it suffices to prove in irreducible representations of $S$ without non-trivial $u_1$ or $u_2$-fixed vectors. By Lemma \ref{le:15}, we consider
$(\beta^\delta_{s_0},\,\mathbb{E}_{s_0}^\delta)$, $s_0\in\RR\backslash 0$.

Let $G_1$ be the connected subgroup with Lie algebra generated by $\{X,\,u_1\}$.  From Lemma \ref{le:15} we see that the restricted representation of $\beta^\delta_{s_0}$ on $G_1$ is irreducible and is exactly $\beta^\delta$. Then it follows from \eqref{for:52} of Theorem \ref{th:2} that $\mathcal{D}_{\delta,m}(g)=0$; moreover, \eqref{for:53} of Theorem \ref{th:2} shows that
$f\cdot r^{-s\lambda_1}\in \mathbb{E}_{s_0}^\delta$. From \eqref{for:58} we see that $f\cdot r^{-s\lambda_2}\in \mathbb{E}_{s_0}^\delta$, which implies that $u_2^sf\in \mathbb{E}_{s_0}^\delta$.

Let $G_2$ be the connected subgroup with Lie algebra generated by $\{X,\,u_2\}$. It is clear that the restricted representation has no nontrivial $u_2$-fixed vectors. Then the estimates of $u_2^sf$ follow from Corollary \ref{cor:1}.

\end{proof}

\section{Proof of Theorem \ref{th:6}}
We recall notations at the beginning of Section \ref{sec:4}. For any $\phi\in \Phi$, fix a basis $\{Y_{(\phi,1)},\cdots,Y_{(\phi,\dim(\mathfrak{g}^\phi))}\}$ of $\mathfrak{g}^\phi$. By the
decomposition \eqref{for:30} we see that $\{Y_{(\phi,j)}\}$, $\phi\in \Phi$, $1\leq j\leq \dim(\mathfrak{g}^\phi)$ is a basis of $\mathfrak{G}$. We assume that the set
$\{Y_{(\phi,j)}\}$ is inside $\mathfrak{G}^1$. Let $x_0=\min\{\phi:\,\phi>0,\,\phi\in \Phi\}$.

By Lemma \ref{le:12}, the equation
\begin{align*}
(X+m)f=g
\end{align*}
has a unique solution $f\in \mathcal{H}$; moreover,
\begin{align}\label{for:75}
    \norm{v^k f}\leq m^{-1}\norm{v^k g},\qquad 0\leq k\leq s
  \end{align}
 if $v=Y_{(\phi,j)}$, where $\phi(X)=0$, $1\leq j\leq \dim(\mathfrak{g}^\phi)$.

 For any $Y_{(\phi,j)}$ with $\phi\neq 0$, $1\leq j\leq \dim(\mathfrak{g}^\phi)$, we consider the connected subgroup with Lie algebra generated by $\{Y_{(\phi,j)},\,X\}$, which we denote by $G_{\phi,j}$. It is clear that $G_{\phi,j}$ is isomorphic to $\RR\ltimes\RR$.  Then $(\pi,\mathcal{H})$ is also a unitary representation of $G_{\phi,j}$. By Howe-Moore, there is no non-trivial $Y_{(\phi,j)}$-fixed vectors. Hence we can apply previous results to the restricted representation of $\pi$ on $G_{\phi,j}$.

 If $\phi<0$, we consider the restricted representation on $G_{\phi,j}$. It follows from \eqref{for:95} of Corollary \ref{cor:1} that
\begin{align}\label{for:12}
  \norm{(I-Y_{(\phi,j)}^2)^\frac{t}{2}f}\leq Cm^{-1}\norm{g}_t,\qquad 0\leq t\leq s
\end{align}
if $\phi<0$, $1\leq j\leq \dim(\mathfrak{g}^\phi)$.

\textbf{Proof of \eqref{for:88}}:  It follows from \eqref{for:75}.

\smallskip

\textbf{Proof of \eqref{for:81}}: By assumption we note that $s\leq\frac{m}{\phi}$, where $\phi\in\Phi$ with $\phi>0$.  By applying Corollary \ref{cor:1} to the restricted representation on $G_{\phi,j}$ we have
\begin{align}\label{for:80}
  \norm{(I-Y_{(\phi,j)}^2)^\frac{t}{2}f}\leq C(m-t\phi(X))^{-1}\norm{g}_t\leq C(m-ty_0)^{-1}\norm{g}_t,
\end{align}
if $0\leq t<s$.

\eqref{for:75}, \eqref{for:12} and \eqref{for:80} together with Theorem \ref{th:4} show that $f\in \mathcal{H}^{s-2}$ with estimates
 \begin{align*}
   \norm{f}_t\leq C_{t,m}\norm{g}_t
 \end{align*}
  if $0\leq t\leq s-2$.  Hence we prove \eqref{for:81}.

  \medskip
\textbf{Proof of \eqref{for:4}}:  Now we consider the case of $s>\frac{m}{y_0}$. Ar first, we show that $(I-Y_{(y_0,1)}^2)^\frac{s}{2}f\in \mathcal{H}$. We suppose $Y_{y_0,1}$ and $\mathfrak{v}$ are linearly independent. Otherwise, the claim is obvious.  We consider the connected subgroup $S$ with Lie algebra $\{X,Y_{y_0,1},\mathfrak{v}\}$. Note that
$[Y_{\omega,1},\mathfrak{v}]=0$. In fact, for any $u\in \mathfrak{g}^\phi$ with $\phi>0$, we have $[Y_{(y_0,1)},u]=0$. Otherwise, $\phi+y_0\in \Phi$ with
$y_0+\phi>y_0$, which contradicts the assumption. Hence we see that $S$ is isomorphic to $\RR\ltimes\RR^2$ as in Section \ref{sec:10}. By Howe-Moore, we can apply Proposition \ref{po:1} to the
the restricted representation on $S$. Then we have $(I-Y_{(y_0,1)}^2)^\frac{s}{2}f\in \mathcal{H}$.

Next, we consider the restricted representation on $G_{y_0,1}$. Corollary \ref{cor:1} shows that
\begin{align}\label{for:82}
 \norm{(I-Y_{(y_0,1)}^2)^\frac{t}{2}f}&\leq C_{t,m}\norm{g}_{t+\frac{\epsilon}{x_0}}\leq C_{t,m}\norm{g}_{t+\frac{1}{2}},\qquad\text{and}\notag\\
 \norm{(I-Y_{(y_0,1)}^2)^\frac{\ell}{2} f}&\leq C_{\ell,m}\norm{g}_{s},
\end{align}
for any $0\leq t\leq s-\frac{1}{2}$ and $0\leq \ell\leq s$ if we choose $\epsilon$ sufficiently small.

 For any $Y_{(\phi,j)}\neq Y_{(y_0,1)}$ with $\phi>0$ and $1\leq j\leq\dim(\mathfrak{g}^\phi)$
We consider the connected subgroup with Lie algebra generated by $\{Y_{(y_0,1)},\,Y_{(\phi,j)},\,X\}$, which we denote by
$S_{\phi,j}$. The above discussion shows that $S_{\phi,j}$ is is isomorphic to $\RR\ltimes\RR^2$ as in Section \ref{sec:10}. Thanks to Howe-Moore, we can apply Proposition \ref{po:1} to the
the restricted representation on $S_{\phi,j}$. Proposition \ref{po:1} and \eqref{for:82} show that
\begin{align}\label{for:83}
 \norm{(I-Y_{(\phi,j)}^2)^\frac{t}{2}f}&\leq C_{t,m}\norm{g}_{t+\frac{\epsilon}{x_0}}\leq C_{t,m}\norm{g}_{t+\frac{1}{2}},\qquad\text{and}\notag\\
 \norm{(I-Y_{(\phi,j)}^2)^\frac{\ell}{2} f}&\leq C_{s,m}\norm{g}_{s}
\end{align}
for any $0\leq t\leq s-\frac{1}{2}$ and $0\leq \ell\leq s$ if we choose $\epsilon$ sufficiently small.

Then it follows from \eqref{for:75}, \eqref{for:12}, \eqref{for:82}, \eqref{for:83} and Theorem \ref{th:4} that $f\in \mathcal{H}^{s}$ with estimates
 \begin{align*}
   \norm{f}_t\leq C_{t,m}\norm{g}_{t+2},\quad\text{and}\quad \norm{f}_s\leq C_{s,m}\norm{g}_{s}
 \end{align*}
  if $0\leq t\leq s-2$. Hence we get \eqref{for:4}.

   \medskip
\textbf{Proof of \eqref{for:84}}: The result follows immediately from \eqref{for:4}.

 \medskip
\textbf{Proof of \eqref{for:87}}: We note that the $(X-m)$-invariant distributions of the restricted representations are also the invariant distributions of $\pi$. Then we can center on $(X-m)$-invariant distributions of subgroups. We consider the restricted representation on $G_{\omega,1}$.  Then it follows from \eqref{for:86} of Corollary \ref{cor:1} that the twisted cohomological equation $(X-m)f=g$ has a solution $f\in \mathcal{H}$ with $(I-Y_{(y_0,1)}^2)^\frac{s}{2}f\in \mathcal{H}$. Then \eqref{for:4} implies that $f\in \mathcal{H}^{s}$.

\section{Proof of Theorem \ref{th:5}}
For any $\RR$-semisimple $Y\in \mathfrak{G}$, we have the decomposition of $\mathfrak{G}$ for $\text{ad}_Y$:
\begin{align*}
\mathfrak{G}=\sum_{\mu\in\Phi(Y)}\mathfrak{g}_Y^\mu
\end{align*}
where $\Phi(Y)$ is the set of eigenvalues and $\mathfrak{g}_Y^\mu$ is the eigenspace for eigenvalue $\mu$.

Let $\mathbf{g}$ be the subalgebra generated by all $\mathfrak{g}_X^\mu$, $\mu\neq 0$. Then $\mathfrak{g}$ is an ideal in $\mathfrak{G}$. Let
$\GG'=\GG'(X)$ be the connected subgroup with Lie algebra $\mathbf{g}$. If $Y$ is sufficiently close to $X$, then $\Phi(Y)\subset \bigcup_{\mu\in \Phi(X)}(\mu-\epsilon,\mu+\epsilon)$, where
$\epsilon$ sufficiently close to $0$,  Note that $\mathbf{g}$ is an invariant subspace for $\text{ad}_Y$. Then we have the direct sum decomposition:
\begin{align*}
\mathfrak{G}=\sum_{\mu\in\Phi_1(Y)}\mathfrak{l}_\mu+\sum_{\mu\in\Phi_2(Y)}\mathfrak{l}_\mu
\end{align*}
such that
\begin{align*}
\mathbf{g}=\sum_{\mu\in\Phi_1(Y)}\mathfrak{l}_\mu.
\end{align*}
We note that $\Phi_1(Y)$ is the set of eigenvalues of $\text{ad}_Y$  restricted on $\mathbf{g}$ and eigenvalues in $\Phi_2(Y)$ are sufficiently close to $0$ if $Y$ is sufficiently close to $X$.

Set $\Phi'=\{\phi\in \Phi(X): \phi>0\}$, $\Phi''=\{\phi\in \Phi(X): \phi=0\}$ and $\Phi'''=\{\phi\in \Phi(X): \phi<0\}$. We choose $\delta$ sufficiently small such that for any $\RR$-semisimple $Y\in \mathfrak{G}$ with  $|Y-X|\leq\delta$ and $m\in\RR$ with $|m-m_0|\leq\delta$, the followings hold:
\begin{enumerate}

  \item\label{for:101} $\min\{\mu:\mu>0,\mu\in \Phi_1(Y)\}>\frac{1}{2}\min\{\phi:\phi\in \Phi'\}$;

\smallskip
  \item\label{for:13} $\max\{|\mu|:\mu\in \Phi_2(Y)\}<\frac{1}{2}\min\{|\phi|:0\neq\phi\in \Phi(X)\}$;

\smallskip
\item\label{for:98} $\frac{1}{2}m_0\leq m\leq 2m_0$;

 \smallskip
\item \label{for:99} $s\max\{|\mu|:\mu\in \Phi_2(Y)\}<\frac{m_0}{8}$;

\smallskip
\item \label{for:2} $\min\{\frac{m}{\mu}:\mu>0,\mu\in \Phi_1(Y)\}\leq \frac{3}{2}\min\{\frac{m_0}{\phi}:\phi\in \Phi'\}$.

\end{enumerate}
Let $s_0=\frac{3}{2}\min\{\frac{m_0}{\phi}:\phi\in \Phi'\}$.

\medskip
\textbf{Proof of \eqref{for:97}}: It is a direct consequence of \eqref{for:88} of  Theorem \ref{th:6} and assumption \eqref{for:98}.

\medskip
\textbf{Proof of \eqref{for:104}}:  Assumptions \eqref{for:13} and \eqref{for:2} and the choice of $s_0$ show that $s>\min\{\frac{m}{\mu}:\mu>0,\mu\in \Phi(Y)\}$. Then
\eqref{for:104} follows from \eqref{for:87} of Theorem \ref{th:6}.

\medskip
\textbf{Proof of \eqref{for:103}}: Assumptions \eqref{for:13} and \eqref{for:2} and the choice of $s_0$ show that $s>\min\{\frac{m}{\mu}:\mu>0,\mu\in \Phi(Y)\}$. Then by
 \eqref{for:84} of Theorem \ref{th:6}, we see that $f\in \mathcal{H}^s$. It remains to show how to obtain Sobolev estimates of $f$.

We can choose a basis $\{Z_{(\mu,j)}: \mu\in \Phi(Y), 1\leq j\leq \dim(\mathfrak{g}^\mu)\}$ of $\mathfrak{G}$ such that the vectors in $\{Z_{(\mu,j)}\}$ are sufficiently close to the vectors in $\{Y_{(\mu,j)}\}$, see the beginning of the proof of Theorem \ref{th:6}. In fact, we can follow the proof line of Theorem \ref{th:6}.

From \eqref{for:75} and assumption \eqref{for:98} we have
\begin{align}\label{for:100}
    \norm{v^k f}\leq m^{-1}\norm{v^k g}\leq 2m_0^{-1},\qquad 0\leq k\leq s
  \end{align}
 if $v=Z_{(\mu,j)}$, where $\mu=0$, $1\leq j\leq \dim(\mathfrak{g}^\mu)$.

If $\mu\in \Phi(Y)$ with $\mu<0$, from \eqref{for:12} and assumption \eqref{for:98} we have
\begin{align}\label{for:14}
  \norm{(I-Z_{(\mu,j)}^2)^\frac{t}{2}f}\leq Cm_0^{-1}\norm{g}_t,\qquad 0\leq t\leq s
\end{align}
if $\mu<0$, $1\leq j\leq \dim(\mathfrak{g}^\phi)$.

Next, we consider $0<\mu\in \Phi_2(Y)$. Assumption \eqref{for:99} implies that $s<\frac{m}{2\mu}$. We consider the subgroup $G'_{\mu,j}$ with Lie algebra generated by $\{Y,Z_{\mu,j}\}$, $1\leq j\leq \dim(\mathfrak{g}^\mu)$. Then by Corollary \ref{cor:1} and assumption \eqref{for:98} we have
\begin{align}\label{for:102}
  \norm{(I-Z_{(\mu,j)}^2)^\frac{t}{2}f}\leq C(m-t\mu)^{-1}\norm{g}_t\leq Cm_0^{-1}\norm{g}_t,
\end{align}
if $0\leq t\leq s$.

Finally, we consider $0<\mu\in \Phi_1(Y)$.  Let $z_0=\max\{\mu:\mu>0,\mu\in \Phi_1(Y)\}$ and $w_0=\min\{\mu:\mu>0,\mu\in \Phi_1(Y)\}$.  We see that \eqref{for:82} holds if we substitute $Y$ by $Z$, $y_0$ by $z_0$ and $x_0$ by $w_0$ thanks to assumption \eqref{for:101}.

Assumption \eqref{for:101} also implies that \eqref{for:83} holds for all $\phi\in \Phi_1(Y)$, $\mu>0$ if we substitute $Y$ by $Z$, $\phi$ by $\mu$ in $\Phi_1(Y)$ and $x_0$ by $w_0$.
Hence we can also get \eqref{for:103} by the Sobolev estimates we obtained so far and Theorem \ref{th:4}.

\section{Proof of Theorem \ref{th:7}}
At first, we consider the case of $\mathfrak{u}$ nilpotent. Since $[X,\mathfrak{u}]=0$, it is clear that $\mathfrak{g}^\phi$ is invariant under $\text{ad}_{\mathfrak{u}}$ for any $\phi\in \Phi$. Since $\text{ad}_{\mathfrak{u}}$ is nilpotent on $\mathfrak{g}^\phi$, there exists $0\neq\mathfrak{v}_\phi$ such that $\text{ad}_{\mathfrak{u}}(\mathfrak{v}_\phi)=0$. Especially, we consider $\phi=y_0$. Then we consider the connected subgroup $S$ with Lie algebra  generated by $\{X,\mathfrak{v}_{y_0}, \mathfrak{u}\}$. It is clear that $S$ is isomorphic to $(\RR\ltimes \RR)\times\RR$. We consider the restricted representation $\pi$ on $S$. By Howe-Moore, we can apply
Proposition \ref{po:3}. Hence we see that the twisted cocycle equation has a common solution $h\in \mathcal{H}$ with $(I-\mathfrak{v}^2_{y_0})^{\frac{s}{2}}h\in \mathcal{H}$. Then \eqref{for:4} of Theorem \ref{th:6} shows that
$h\in \mathcal{H}^s$ and the estimates follow immediately.

If $\mathfrak{u}$ is in a $\RR$-split Cartan algebra and $[X,\mathfrak{u}]=0$, then there exists $0\neq\mathfrak{v}\in \mathfrak{g}^\phi$ such that $[\mathfrak{u},\mathfrak{v}]=\lambda\mathfrak{v}$. Here $\phi(X)=y_0$.
If $\lambda=0$, we go back to the nilpotent case, where the connected subgroup with Lie algebra  $\{X,\mathfrak{v}, \mathfrak{u}\}$ is isomorphic to $(\RR\ltimes \RR)\times\RR$.

If $\lambda\neq0$, we can rewrite the twisted
cocycle equation as
\begin{align*}
 (X+m-\lambda^{-1}\phi(X)(\mathfrak{u}+m_1))g_2=(X+m)(g_2-\lambda^{-1}\phi(X)g_1).
\end{align*}
It is clear that $X-\lambda^{-1}\phi(X)\mathfrak{u}$ is $\RR$-semisimple and $[X-\lambda^{-1}\phi(X)\mathfrak{u},\mathfrak{v}]=0$. Then also we go back to the nilpotent case, where the connected subgroup with Lie algebra  $\{X,\mathfrak{v}, X-\lambda^{-1}\phi(X)\mathfrak{u}\}$ is isomorphic to $(\RR\ltimes \RR)\times\RR$. Then we have a common solution
$h\in \mathcal{H}^s$ which solves
\begin{align*}
 (X+m)h&=g_2\\
 (X+m-\lambda^{-1}\phi(X)(\mathfrak{u}+m_1))h&=g_2-\lambda^{-1}\phi(X)g_1
\end{align*}
simultaneously, which is exactly
\begin{align*}
 &(X+m)h=g_2, \qquad (\mathfrak{u}+m_1)h=g_1.
\end{align*}
The estimates of $h$ follows from Theorem \ref{th:6}. Hence we finish the proof.

\end{document}